\documentclass{article}

\usepackage[english,portuges]{babel}

\usepackage{amsmath,amssymb,amsthm}
\usepackage[latin1]{inputenc}
\usepackage[T1]{fontenc}

\usepackage{graphicx}

\usepackage{url}

\usepackage{eepic,epsfig}

\usepackage{hyperref}

\newcommand{\xiipt}{}

\newtheorem{theorem}{Teorema}
\newtheorem{corollary}[theorem]{Corol\'{a}rio}
\newtheorem{proposition}[theorem]{Proposi\c{c}\~{a}o}

\theoremstyle{remark}
\newtheorem{remark}[theorem]{Observa\c{c}\~{a}o}
\newtheorem*{problema}{Problema}

\theoremstyle{definition}
\newtheorem{definition}[theorem]{Defini\c{c}\~{a}o}

\def\R{\mathbb{R}}

\begin{document}

\selectlanguage{portuges}

\title{O controlo \'{o}ptimo e as suas m\'{u}ltiplas aplica\c{c}\~{o}es\thanks{Dedicado
a Francis Clarke e a Richard Vinter por ocasi\~{a}o da celebra\c{c}\~{a}o
do sexag\'{e}simo anivers\'{a}rio de ambos os matem\'{a}ticos,
\emph{Workshop in Control, Nonsmooth Analysis
and Optimization}, Porto, 4 a 8 de Maio de 2009
\href{http://ceoc.mat.ua.pt/fc-rv-60}{\texttt{<http://ceoc.mat.ua.pt/fc-rv-60>}.}}}

\author{Cristiana J. Silva$^{1,2}$\\
\href{mailto:cjoaosilva@ua.pt}{\url{cjoaosilva@ua.pt}}
\and Delfim F. M. Torres$^{1}$\\
\href{mailto:delfim@ua.pt}{\url{delfim@ua.pt}}
\and Emmanuel Tr\'elat$^{2}$\\
\href{mailto:emmanuel.trelat@univ-orleans.fr}{\url{emmanuel.trelat@univ-orleans.fr}}}

\date{$^{1}$Control theory group (cotg)\\
Centro de Estudos em Optimiza\c{c}\~{a}o e Controlo (CEOC)\\
Departamento de Matem\'{a}tica, Universidade de Aveiro\\
3810-193 Aveiro, Portugal\\[0.3cm]
$^{2}$Universit\'{e} d'Orl\'{e}ans, UFR Sciences\\
F\'{e}d\'{e}ration Denis Poisson\\
Math\'{e}matiques, Laboratoire MAPMO, UMR 6628\\
45067 Orl\'{e}ans Cedex 2, France}

\maketitle


\begin{abstract}
Neste trabalho s\~{a}o referidas motiva\c{c}\~{o}es,
aplica\c{c}\~{o}es e rela\c{c}\~{o}es da teoria do controlo
com outras \'{a}reas da matem\'{a}tica. Apresentamos uma breve
resenha hist\'{o}rica sobre o controlo \'{o}ptimo, desde as suas
origens no c\'{a}lculo das varia\c{c}\~{o}es e na teoria cl\'{a}ssica do controlo
aos dias de hoje, dando especial destaque
ao princ\'{\i}pio do m\'{a}ximo de Pontryagin.

\smallskip

\noindent \textbf{Palavras chave:} controlo \'{o}ptimo,
princ\'{\i}pio do m\'{a}ximo de Pontryagin, aplica\c{c}\~{o}es
da teoria matem\'{a}tica dos sistemas de controlo.
\end{abstract}

\selectlanguage{english}

\begin{abstract}
In this work we refer to motivations,
applications, and relations of control theory
with other areas of mathematics.
We present a brief historical review of optimal control theory,
from its roots in the calculus of variations and the classical theory
of control to the present time, giving particular emphasis
to the Pontryagin maximum principle.

\smallskip


\smallskip

\end{abstract}

\selectlanguage{portuges}


\section{Introdu\c{c}\~{a}o}

Todos n\'{o}s j\'{a} tent\'{a}mos, numa ou outra ocasi\~{a}o,
manter em equil\'{\i}brio uma vara sobre o dedo indicador
(\textrm{i.e.}, resolver o problema do p\^{e}ndulo invertido).
Por outro lado \'{e} muito mais dif\'{\i}cil, sobretudo se fecharmos os olhos,
manter em equil\'{\i}brio um p\^{e}ndulo invertido duplo.
A teoria do controlo permite faz\^{e}-lo sob a condi\c{c}\~{a}o de dispormos
de um bom modelo matem\'{a}tico.

Um sistema de controlo \'{e} um sistema din\^{a}mico, que evolui no tempo,
sobre o qual podemos agir atrav\'{e}s de uma fun\c{c}\~{a}o de comando ou controlo.
Um computador, que permite a um utilizador efectuar uma s\'{e}rie de comandos,
um ecossistema sobre o qual podemos agir favorecendo esta ou aquela esp\'{e}cie,
os tecidos nervosos que formam uma rede controlada pelo c\'{e}rebro e realizam
a transforma\c{c}\~{a}o de est\'{\i}mulos provenientes do exterior em ac\c{c}\~{o}es do organismo,
um robot que deve efectuar uma tarefa bem precisa, uma viatura sobre
a qual agimos por interm\'{e}dio de um pedal de acelera\c{c}\~{a}o,
de travagem e embraiagem e que conduzimos com a ajuda de um volante,
um sat\'{e}lite ou uma nave espacial, s\~{a}o todos eles exemplos de sistemas
de controlo, os quais podem ser modelados e estudados
pela teoria dos sistemas de controlo.

A teoria do controlo analisa as propriedades de tais sistemas,
com o intuito de os ``conduzir'' de um determinado estado inicial
a um dado estado final, respeitando eventualmente certas restri\c{c}\~{o}es.
A origem de tais sistemas pode ser muito diversa:
mec\^{a}nica, el\'{e}ctrica, biol\'{o}gica, qu\'{\i}mica, econ\'{o}mica, etc.
O objectivo pode ser o de estabilizar o sistema tornando-o
insens\'{\i}vel a certas perturba\c{c}\~{o}es (problema de \emph{estabiliza\c{c}\~{a}o})
ou ainda determinar as solu\c{c}\~{o}es \'{o}ptimas relativamente
a um determinado crit\'{e}rio de optimiza\c{c}\~{a}o
(problema do \emph{controlo \'{o}ptimo}).
Para modelar os sistemas de controlo podemos recorrer a equa\c{c}\~{o}es diferenciais,
integrais, funcionais, de diferen\c{c}as finitas, \`{a}s derivadas parciais,
determin\'{\i}sticas ou estoc\'{a}sticas, etc.
Por esta raz\~{a}o a teoria do controlo vai beber e contribui
em numerosos dom\'{\i}nios da matem\'{a}tica (\textrm{vide}, \textrm{e.g.},
\cite{Bryson,Coron_Trelat_2004,Coron_Trelat_2007,MR0638591,McShane,Sontag}).

A estrutura de um sistema de controlo \'{e} representada pela interconex\~{a}o
de certos elementos mais simples que formam sub-sistemas.
Neles transita \emph{informa\c{c}\~{a}o}. A din\^{a}mica de um sistema de controlo
define as transforma\c{c}\~{o}es poss\'{\i}veis do sistema,
que ocorrem no tempo de maneira determinista ou aleat\'{o}ria.
Os exemplos j\'{a} dados mostram que a estrutura e a din\^{a}mica
de um sistema de controlo podem ter significados muito diferentes.
Em particular, o conceito de sistema de controlo pode descrever
transforma\c{c}\~{o}es discretas, cont\'{\i}nuas, h\'{\i}bridas ou,
de um modo mais geral, numa \emph{time scale} ou
\emph{measure chain} \cite{comRui:TS:Lisboa07,311,Naty}.

Um sistema de controlo diz-se \emph{control\'{a}vel} se o podemos ``conduzir''
(em tempo finito) de um determinado estado inicial at\'{e} um estado final prescrito.
Em rela\c{c}\~{a}o ao problema da controlabilidade, Kalman demonstrou em 1949
um resultado importante que caracteriza os sistemas lineares control\'{a}veis
de dimens\~{a}o finita (Teorema~\ref{condKalman}).
Para sistemas n\~{a}o lineares o problema matem\'{a}tico da controlabilidade
\'{e} muito mais dif\'{\i}cil e constitui um dom\'{\i}nio de investiga\c{c}\~{a}o ainda activo nos dias de hoje.

Assegurada a propriedade de controlabilidade, podemos desejar passar
de um estado inicial a um estado final minimizando
ou maximizando um determinado crit\'{e}rio. Temos ent\~{a}o um problema
de controlo \'{o}ptimo. Por exemplo, um condutor que efectue
o trajecto Lisboa-Porto pode querer viajar em tempo m\'{\i}nimo.
Nesse caso escolhe o trajecto pela auto-estrada A1. Uma consequ\^{e}ncia
de tal escolha ser\'{a} o pagamento de portagem. Outro problema
de controlo \'{o}ptimo \'{e} obtido se tivermos como crit\'{e}rio de
minimiza\c{c}\~{a}o os custos da viagem. A solu\c{c}\~{a}o de tal problema
envolver\'{a} a escolha de estradas nacionais, gratuitas,
mas que levam muito mais tempo a chegar ao destino
(segundo a informa\c{c}\~{a}o do s\'{\i}tio da internet \url{http://www.google.pt/maps}
o trajecto pela auto-estrada dura 3h
e pela estrada nacional dura 6h45m).
Um problema de controlo \'{o}ptimo pode ser formulado do seguinte modo.
Consideremos um sistema de controlo, cujo estado num determinado instante
\'{e} representado por um vector. Os controlos s\~{a}o fun\c{c}\~{o}es ou par\^{a}metros,
habitualmente sujeitos a restri\c{c}\~{o}es, que agem sobre o sistema
sob a forma de for\c{c}as exteriores,
de potenciais t\'{e}rmicos ou el\'{e}ctricos, de programas de investimento, etc.
e afectam a din\^{a}mica. Uma equa\c{c}\~{a}o \'{e} dada,
ou tipicamente um sistema de equa\c{c}\~{o}es diferenciais,
relacionando as vari\'{a}veis e modelando a din\^{a}mica do sistema.
\'{E} depois necess\'{a}rio utilizar a informa\c{c}\~{a}o presente e as caracter\'{\i}sticas
do problema para construir os controlos adequados
que v\~{a}o permitir realizar um objectivo preciso.
Por exemplo, quando nos deslocamos na nossa viatura agimos
de acordo com o c\'{o}digo da estrada (pelo menos \'{e} aconselh\'{a}vel)
e concretizamos um plano de viagem para chegar ao nosso destino.
S\~{a}o impostas restri\c{c}\~{o}es sobre a traject\'{o}ria ou sobre os controlos,
que \'{e} imprescind\'{\i}vel ter em considera\c{c}\~{a}o.
Fixamos um crit\'{e}rio permitindo medir a qualidade do processo escolhido.
Este apresenta-se normalmente sob a forma de uma funcional
que depende do estado do sistema e dos controlos.
Para al\'{e}m das condi\c{c}\~{o}es anteriores procuramos ainda minimizar (ou maximizar)
esta quantidade. Um exemplo j\'{a} dado anteriormente \'{e} o de deslocarmo-nos
em tempo m\'{\i}nimo de um ponto a outro. Notemos que a forma das traject\'{o}rias \'{o}ptimas
depende fortemente do crit\'{e}rio de optimiza\c{c}\~{a}o. Por exemplo,
para estacionar o nosso carro \'{e} f\'{a}cil verificar que a traject\'{o}ria
seguida difere se queremos realizar a opera\c{c}\~{a}o em tempo m\'{\i}nimo
(o que \'{e} arriscado) ou minimizando a quantidade de combust\'{\i}vel gasta na opera\c{c}\~{a}o.

\begin{figure}[!ht]
\centering
\includegraphics[width=5cm]{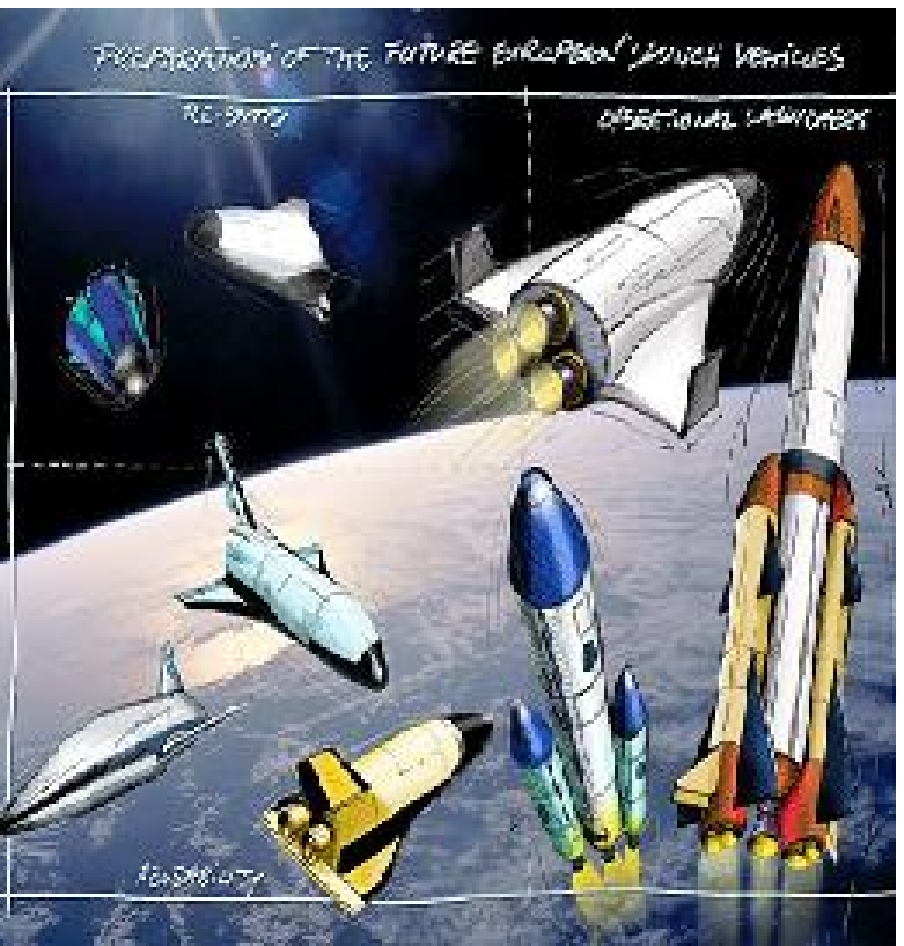}
\includegraphics[width=5cm]{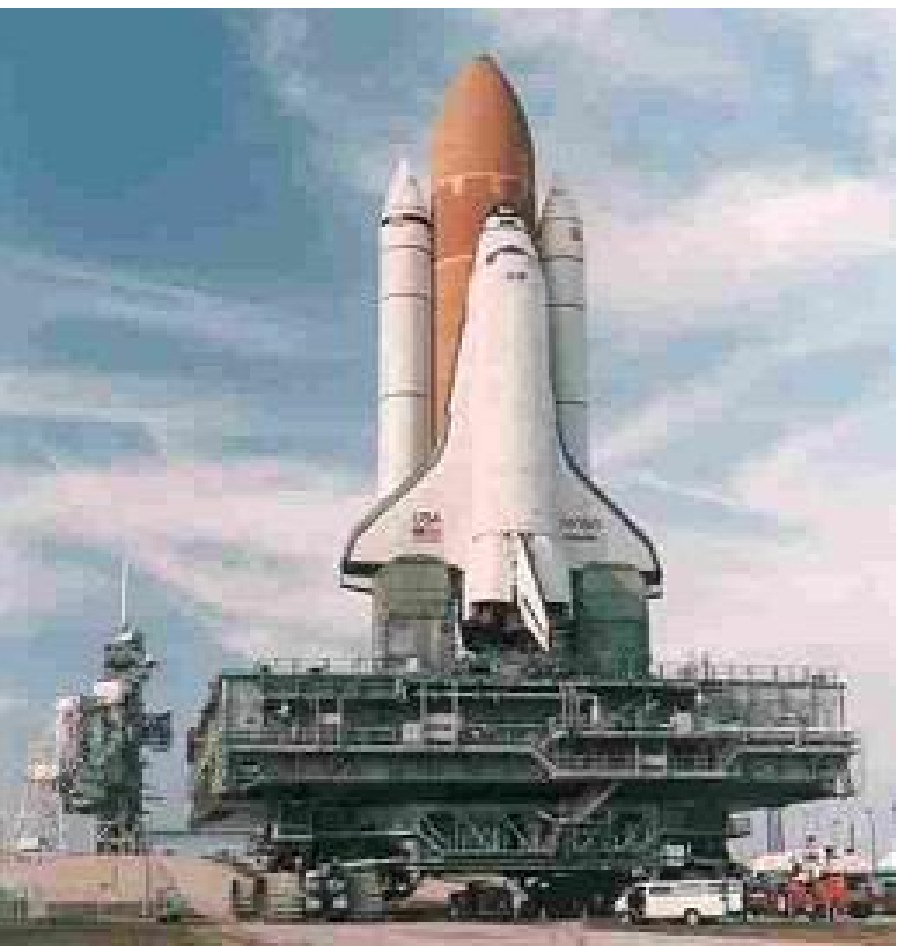}
\caption{a teoria do controlo \'{o}ptimo tem um papel importante
na engenharia aeroespacial.}
\label{espaco}
\end{figure}

A teoria do controlo \'{o}ptimo tem uma grande import\^{a}ncia
no dom\'{\i}nio aeroespacial, nomeadamente em problemas de condu\c{c}\~{a}o,
transfer\^{e}ncia de \'{o}rbitas aero-assistidas, desenvolvimento
de lan\c{c}adores de sat\'{e}lites recuper\'{a}veis (o aspecto financeiro
\'{e} aqui muito importante) e problemas da reentrada atmosf\'{e}rica,
como seja o famoso projecto \emph{Mars Sample Return}
da Ag\^{e}ncia Espacial Europeia (ESA), que consiste em enviar
uma nave espacial ao planeta Marte com o objectivo
de trazer amostras marcianas (Figura~\ref{espaco}).


\section{Breve resenha hist\'{o}rica}

O c\'{a}lculo das varia\c{c}\~{o}es nasceu no s\'{e}culo dezassete com o contributo
de Bernoulli, Fermat, Leibniz e Newton.
Alguns matem\'{a}ticos como H.J.~Sussmann e J.C.~Willems defendem
a origem do controlo \'{o}ptimo coincidente com o nascimento
do c\'{a}lculo das varia\c{c}\~{o}es, em 1697, data de publica\c{c}\~{a}o
da solu\c{c}\~{a}o do problema da braquist\'{o}crona pelo matem\'{a}tico Johann Bernoulli
\cite{Sussmann}. Outros v\~{a}o ainda mais longe, chamando a aten\c{c}\~{a}o
para o facto do problema da resist\^{e}ncia aerodin\^{a}mica de Newton,
colocado e resolvido por Isaac Newton em 1686,
no seu \emph{Principia Mathematica}, ser um verdadeiro
problema de Controlo \'{O}ptimo \cite{comCristianaCC,NewtonTypeProb}.

Em 1638 Galileu estudou o seguinte problema: determinar a curva sobre
a qual uma pequena esfera rola sob a ac\c{c}\~{a}o da gravidade,
sem velocidade inicial e sem atrito, de um ponto $A$ at\'{e} um ponto $B$
com um tempo de percurso m\'{\i}nimo (escorrega de tempo m\'{\i}nimo,
ver Figura~\ref{pontoAeB}).

\begin{figure}[!htb]
\centering
\includegraphics[width=5cm]{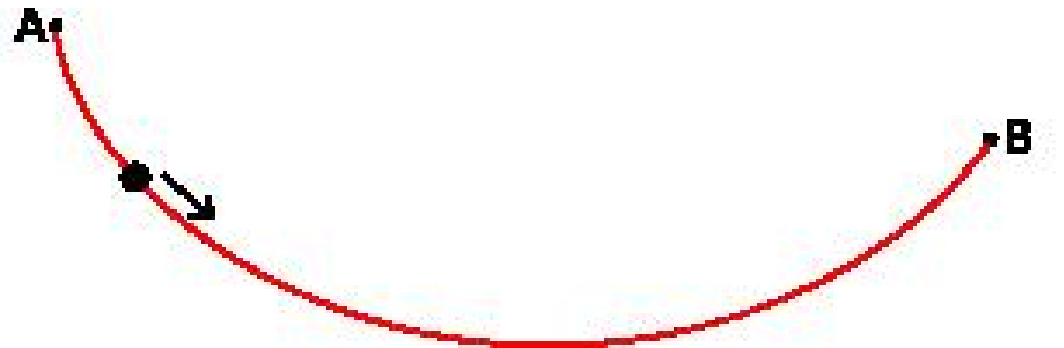}
\includegraphics[width=5cm]{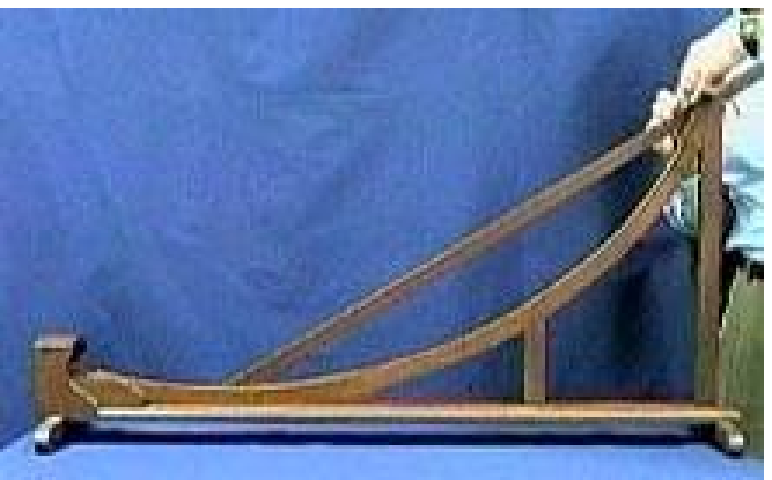}
\caption{problema da braquist\'{o}crona.}
\label{pontoAeB}
\end{figure}

Trata-se do problema da braquist\'{o}crona (do grego \emph{brakhistos}, ``o mais breve'',
e \emph{chronos}, ``tempo''). Galileu pensou (erradamente) que a curva procurada
era um arco de c\'{\i}rculo. Observou no entanto, correctamente, que o segmento
de linha recta n\~{a}o \'{e} o caminho de tempo mais curto. Em 1696,
Jean Bernoulli colocou este problema como um desafio
aos melhores matem\'{a}ticos da sua \'{e}poca. Ele pr\'{o}prio encontrou a solu\c{c}\~{a}o,
assim como o seu irm\~{a}o Jacques Bernoulli, Newton, Leibniz
e o marqu\^{e}s de l'Hopital. A solu\c{c}\~{a}o \'{e} um arco de cicl\'{o}ide come\c{c}ando
com uma tangente vertical \cite{comAndreia,Sussmann}.
As rampas de skate assim como as descidas mais r\'{a}pidas
dos \emph{aqua-parques},
t\^{e}m a forma de cicl\'{o}ide (Figura~\ref{cicloide}).

\begin{figure}[!htb]
\centering
\includegraphics[width=5cm]{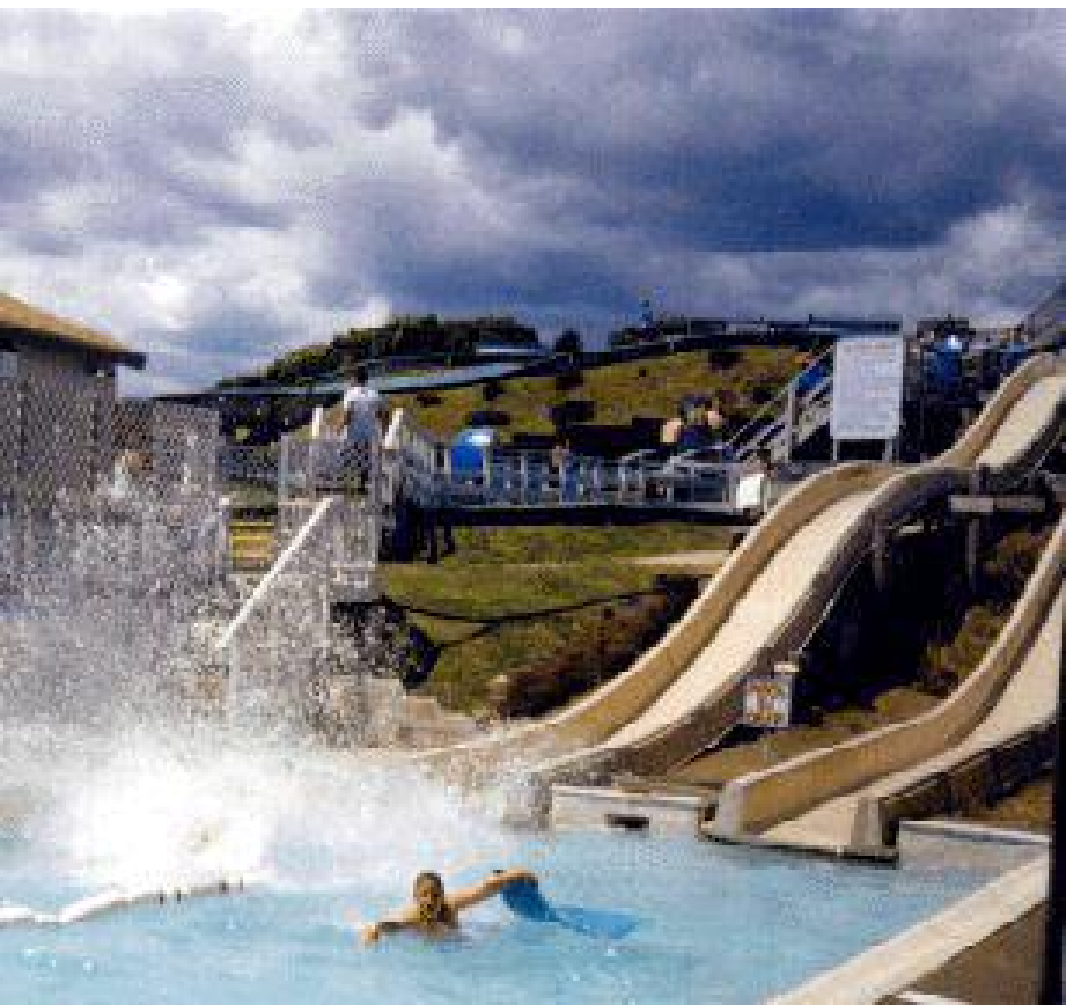}
\includegraphics[width=5cm]{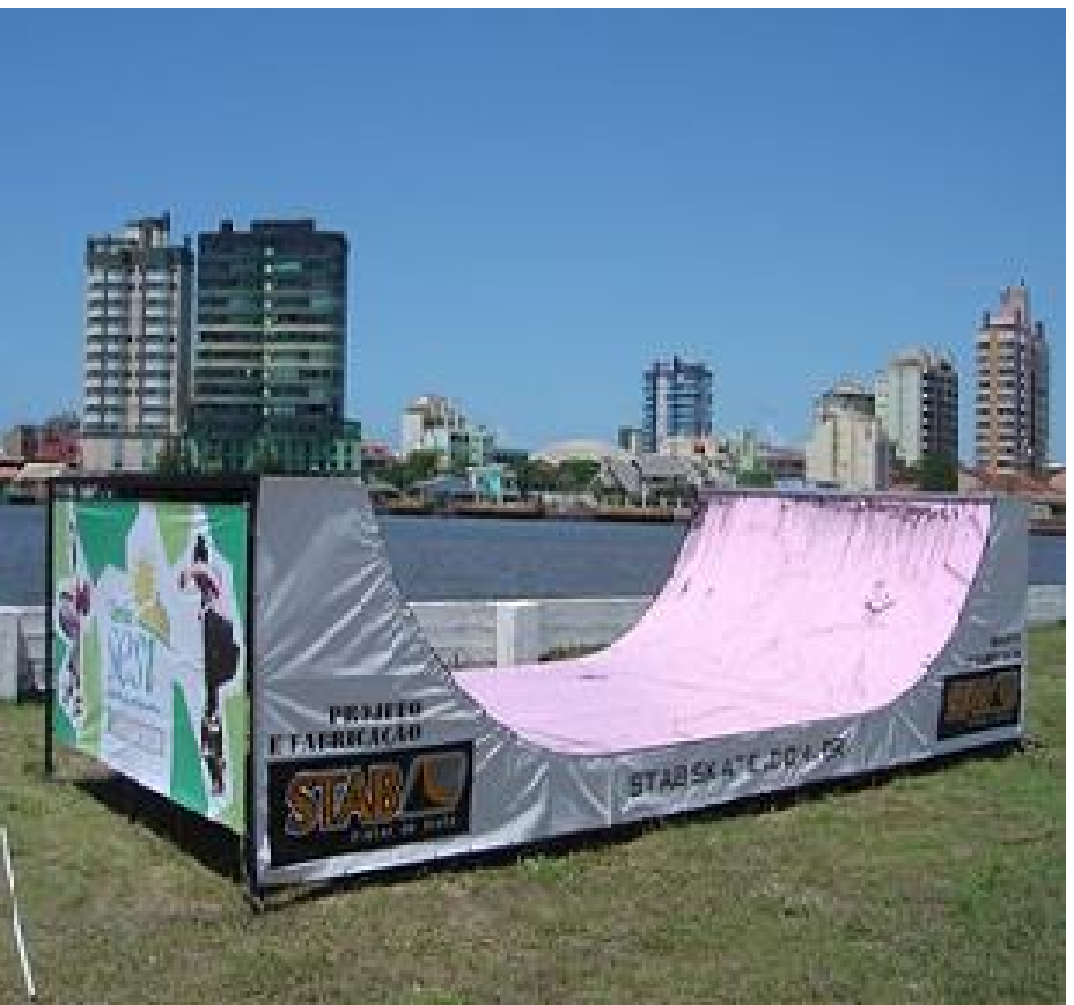}
\caption{arcos de cicl\'{o}ide conduzem \`{a}s descidas mais r\'{a}pidas
e \`{a} adrenalina m\'{a}xima.}
\label{cicloide}
\end{figure}

A teoria do controlo \'{o}ptimo surge depois da segunda guerra mundial,
respondendo a necessidades pr\'{a}ticas de engenharia,
nomeadamente no dom\'{\i}nio da aeron\'{a}utica e da din\^{a}mica de voo.
A formaliza\c{c}\~{a}o desta teoria colocou v\'{a}rias quest\~{o}es novas.
Por exemplo, a teoria do controlo \'{o}ptimo motivou a introdu\c{c}\~{a}o
de novos conceitos de solu\c{c}\~{o}es generalizadas
na teoria das equa\c{c}\~{o}es diferenciais e originou novos resultados
de exist\^{e}ncia de traject\'{o}rias.
Regra geral, considera-se que a teoria do controlo \'{o}ptimo surgiu
em finais dos anos cinquenta na antiga Uni\~{a}o Sovi\'{e}tica, em 1956,
com a formula\c{c}\~{a}o e demonstra\c{c}\~{a}o do Princ\'{\i}pio do M\'{a}ximo de
Pontryagin por L.S.~Pontryagin (Figura~\ref{pontryagin}) e pelo
seu grupo de colaboradores: V.G.~Boltyanskii, R.V.~Gamkrelidze
e E.F.~Mishchenko \cite{Pontryagin_et_all}.

\begin{figure}[!htb]
\centering
\includegraphics[scale=0.4]{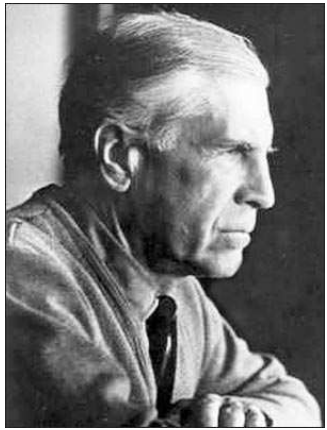}
\caption{Lev Semenovich Pontryagin
(3/Set/1908 -- 3/Maio/1988)}
\label{pontryagin}
\end{figure}

Pontryagin e os seus companheiros introduziram um aspecto
de import\^{a}ncia primordial: generalizaram a teoria
do c\'{a}lculo das varia\c{c}\~{o}es a curvas que tomam valores
em conjuntos fechados (com fronteira). A teoria do controlo \'{o}ptimo
est\'{a} muito ligada \`{a} mec\^{a}nica cl\'{a}ssica, em particular
aos princ\'{\i}pios variacionais (princ\'{\i}pio de Fermat,
equa\c{c}\~{o}es de Euler-Lagrange, etc.)
Na verdade o princ\'{\i}pio do m\'{a}ximo de Pontryagin \'{e} uma generaliza\c{c}\~{a}o
das condi\c{c}\~{o}es necess\'{a}rias de Euler-Lagrange e de Weierstrass.
Alguns pontos fortes da nova teoria foram a descoberta do m\'{e}todo
de programa\c{c}\~{a}o din\^{a}mica, a introdu\c{c}\~{a}o da an\'{a}lise funcional
na teoria dos sistemas \'{o}ptimos e a descoberta de liga\c{c}\~{o}es
entre as solu\c{c}\~{o}es de um problema de controlo \'{o}ptimo
e os resultados da teoria de estabilidade de Lyapunov
\cite{Trelat_2008a,Trelat_2008}. Mais tarde apareceram
as funda\c{c}\~{o}es da teria do controlo estoc\'{a}stico e da
filtragem em sistemas din\^{a}micos, a teoria dos jogos,
o controlo de equa\c{c}\~{o}es com derivadas parciais e os sistemas
de controlo h\'{\i}brido -- algumas de entre as muitas \'{a}reas
de investiga\c{c}\~{a}o actual \cite{MR2005b:93002,leitao,Sontag}.


\section{Controlo \'{o}ptimo linear}
\label{sec:cont:linear}

A teoria do controlo \'{o}ptimo \'{e} muito mais simples quando
o sistema de controlo sob considera\c{c}\~{a}o \'{e} linear.
O controlo \'{o}ptimo n\~{a}o linear ser\'{a} abordado
na Sec\c{c}\~{a}o~\ref{sec:CO:NL}. A teoria linear ainda \'{e},
nos dias de hoje, a mais usada e conhecida nas \'{a}reas
de engenharia e suas aplica\c{c}\~{o}es.


\subsection{Quest\~{o}es centrais}

Seja $A \in \mathcal{M}_n(\R)$
(denotamos por $\mathcal{M}_n(\R)$
o conjunto das matrizes $n \times n$
de entradas reais); $B, X_0 \in \mathcal{M}_{n,1}(\R) \simeq \R^n$;
$I$ um intervalo de $\R$; e $u: \R \to \R$ uma fun\c{c}\~{a}o mensur\'{a}vel ($u \in L^1$) tal que
$u(t) \in I$ $\forall t$.\footnote{Nas aplica\c{c}\~{o}es
considera-se normalmente como classe dos controlos admiss\'{\i}veis
o conjunto dos controlos seccionalmente cont\'{\i}nuos ou mesmo seccionalmente constantes.
Mostra-se que a fam\'{\i}lia de traject\'{o}rias correspondentes aos controlos seccionalmente
constantes \'{e} densa no conjunto de todas as solu\c{c}\~{o}es com controlos mensur\'{a}veis
(\textrm{vide}, \textrm{e.g.}, \cite{livro:Bressan:Piccoli}).} O teorema de exist\^{e}ncia
de solu\c{c}\~{a}o para equa\c{c}\~{o}es diferenciais assegura a exist\^{e}ncia de uma \'{u}nica aplica\c{c}\~{a}o
$\R \ni t \mapsto X(t) \in \R^n$ absolutamente cont\'{\i}nua
($X \in AC$) tal que
\begin{equation}
\label{contsystlin}
\begin{split}
\dot{X}(t) &= A X(t) + B u(t) \quad \forall t \, ,\\
X(0) &= X_0 \, .
\end{split}
\end{equation}
Esta aplica\c{c}\~{a}o depende do controlo $u$. Ao mudarmos a fun\c{c}\~{a}o $u$
obtemos uma outra traject\'{o}ria $t \mapsto X(t)$ em $\R^n$
(Figura~\ref{trajectoria}).

\begin{figure}[!htb]
\begin{center}
\unitlength=.25mm
\makeatletter
\def\shade{\@ifnextchar[{\shade@special}{\@killglue\special{sh}\ignorespaces}}
\def\shade@special[#1]{\@killglue\special{sh #1}\ignorespaces}
\makeatother
\begin{picture}(227,161)(215,-5)
\thinlines
\typeout{\space\space\space eepic-ture exported by 'qfig'.}
\font\FonttenBI=cmbxti10\relax
\font\FonttwlBI=cmbxti10 scaled \magstep1\relax
\path (296,72)(303,66)
\path (296,66)(303,73)
\path (299,69)(300.59,69.31)(302.16,69.64)(303.71,69.99)(305.24,70.36)
(306.75,70.75)(308.24,71.16)(309.71,71.59)(311.16,72.04)(312.59,72.51)
(314,73)(315.165,72.925)(316.36,73)(317.585,73.225)(318.84,73.6)
(320.125,74.125)(321.44,74.8)(322.785,75.625)(324.16,76.6)(325.565,77.725)
(327,79)(328.87,80.785)(330.68,82.64)(332.43,84.565)(334.12,86.56)
(335.75,88.625)(337.32,90.76)(338.83,92.965)(340.28,95.24)(341.67,97.585)
(343,100)(343.55,103.205)(344.2,106.32)(344.95,109.345)(345.8,112.28)
(346.75,115.125)(347.8,117.88)(348.95,120.545)(350.2,123.12)(351.55,125.605)
(353,128)(354.91,130.575)(356.84,133)(358.79,135.275)(360.76,137.4)
(362.75,139.375)(364.76,141.2)(366.79,142.875)(368.84,144.4)(370.91,145.775)
(373,147)(375.11,148.075)(377.24,149)(379.39,149.775)(381.56,150.4)
(383.75,150.875)(385.96,151.2)(388.19,151.375)(390.44,151.4)(392.71,151.275)
(395,151)
\path (299,70)(305.56,68.575)(311.84,67.4)(317.84,66.475)(323.56,65.8)
(329,65.375)(334.16,65.2)(339.04,65.275)(343.64,65.6)(347.96,66.175)
(352,67)(353.69,69.74)(355.56,72.36)(357.61,74.86)(359.84,77.24)
(362.25,79.5)(364.84,81.64)(367.61,83.66)(370.56,85.56)(373.69,87.34)
(377,89)(382.245,90.675)(387.28,92.2)(392.105,93.575)(396.72,94.8)
(401.125,95.875)(405.32,96.8)(409.305,97.575)(413.08,98.2)(416.645,98.675)
(420,99)(423.145,99.175)(426.08,99.2)(428.805,99.075)(431.32,98.8)
(433.625,98.375)(435.72,97.8)(437.605,97.075)(439.28,96.2)(440.745,95.175)
(442,94)
\path (298,70)(298.35,67.785)(298.8,65.64)(299.35,63.565)(300,61.56)
(300.75,59.625)(301.6,57.76)(302.55,55.965)(303.6,54.24)(304.75,52.585)
(306,51)(308.07,49.755)(310.08,48.52)(312.03,47.295)(313.92,46.08)
(315.75,44.875)(317.52,43.68)(319.23,42.495)(320.88,41.32)(322.47,40.155)
(324,39)(325.29,38.395)(326.56,37.68)(327.81,36.855)(329.04,35.92)
(330.25,34.875)(331.44,33.72)(332.61,32.455)(333.76,31.08)(334.89,29.595)
(336,28)(337.09,26.295)(338.16,24.48)(339.21,22.555)(340.24,20.52)
(341.25,18.375)(342.24,16.12)(343.21,13.755)(344.16,11.28)(345.09,8.695)
(346,6)
\path (298,70)(296.01,72.105)(294.04,74.12)(292.09,76.045)(290.16,77.88)
(288.25,79.625)(286.36,81.28)(284.49,82.845)(282.64,84.32)(280.81,85.705)
(279,87)(277.3,87.98)(275.6,88.92)(273.9,89.82)(272.2,90.68)
(270.5,91.5)(268.8,92.28)(267.1,93.02)(265.4,93.72)(263.7,94.38)
(262,95)(260.885,95.715)(259.64,96.36)(258.265,96.935)(256.76,97.44)
(255.125,97.875)(253.36,98.24)(251.465,98.535)(249.44,98.76)(247.285,98.915)
(245,99)(242.585,99.015)(240.04,98.96)(237.365,98.835)(234.56,98.64)
(231.625,98.375)(228.56,98.04)(225.365,97.635)(222.04,97.16)(218.585,96.615)
(215,96)
\path (299,70)(294.86,69.025)(290.84,68)(286.94,66.925)(283.16,65.8)
(279.5,64.625)(275.96,63.4)(272.54,62.125)(269.24,60.8)(266.06,59.425)
(263,58)(259.52,56.525)(256.28,55)(253.28,53.425)(250.52,51.8)
(248,50.125)(245.72,48.4)(243.68,46.625)(241.88,44.8)(240.32,42.925)
(239,41)(238.415,38.665)(237.96,36.36)(237.635,34.085)(237.44,31.84)
(237.375,29.625)(237.44,27.44)(237.635,25.285)(237.96,23.16)(238.415,21.065)
(239,19)(239.715,16.965)(240.56,14.96)(241.535,12.985)(242.64,11.04)
(243.875,9.125)(245.24,7.24)(246.735,5.385)(248.36,3.56)(250.115,1.765)
(252,0)
\put(290,78){{\xiipt\rm {$X_0$}}}
\end{picture}
\end{center}
\caption{a traject\'{o}ria solu\c{c}\~{a}o do sistema de controlo
\eqref{contsystlin} depende da escolha concreta do controlo $u$.}
\label{trajectoria}
\end{figure}
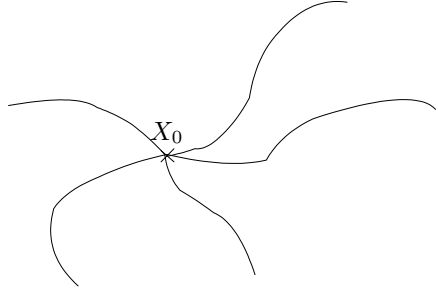

Neste contexto, surgem naturalmente duas quest\~{o}es:

(i) Dado um ponto $X_1 \in \R^n$, existir\'{a} um controlo $u$
tal que a traject\'{o}ria associada a esse controlo liga $X_0$ a $X_1$
em tempo finito $T$? (Figura~\ref{probcontrolabilidade})
\'{E} este o \emph{problema da controlabilidade}.

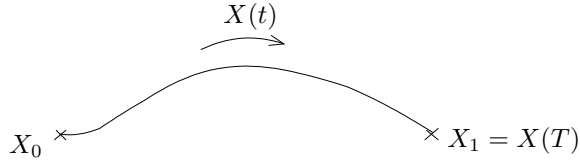
\begin{figure}[!htb]
\begin{center}
\unitlength=.25mm
\makeatletter
\def\shade{\@ifnextchar[{\shade@special}{\@killglue\special{sh}\ignorespaces}}
\def\shade@special[#1]{\@killglue\special{sh #1}\ignorespaces}
\makeatother
\begin{picture}(321,90)(176,-5)
\thinlines
\typeout{\space\space\space eepic-ture exported by 'qfig'.}
\font\FonttenBI=cmbxti10\relax
\font\FonttwlBI=cmbxti10 scaled \magstep1\relax
\path (201,11)(206,6)
\path (200,6)(206,11)
\path (398,11)(404,6)
\path (397,6)(404,12)
\path (203,9)(204.965,8.76)(206.96,8.64)(208.985,8.64)(211.04,8.76)
(213.125,9)(215.24,9.36)(217.385,9.84)(219.56,10.44)(221.765,11.16)
(224,12)(226.355,13.59)(228.72,15.16)(231.095,16.71)(233.48,18.24)
(235.875,19.75)(238.28,21.24)(240.695,22.71)(243.12,24.16)(245.555,25.59)
(248,27)(250.545,28.66)(253.08,30.24)(255.605,31.74)(258.12,33.16)
(260.625,34.5)(263.12,35.76)(265.605,36.94)(268.08,38.04)(270.545,39.06)
(273,40)(275.31,40.815)(277.64,41.56)(279.99,42.235)(282.36,42.84)
(284.75,43.375)(287.16,43.84)(289.59,44.235)(292.04,44.56)(294.51,44.815)
(297,45)(299.285,45.115)(301.64,45.16)(304.065,45.135)(306.56,45.04)
(309.125,44.875)(311.76,44.64)(314.465,44.335)(317.24,43.96)(320.085,43.515)
(323,43)(326.84,42.145)(330.56,41.28)(334.16,40.405)(337.64,39.52)
(341,38.625)(344.24,37.72)(347.36,36.805)(350.36,35.88)(353.24,34.945)
(356,34)(357.965,33.18)(359.96,32.32)(361.985,31.42)(364.04,30.48)
(366.125,29.5)(368.24,28.48)(370.385,27.42)(372.56,26.32)(374.765,25.18)
(377,24)(379.265,22.78)(381.56,21.52)(383.885,20.22)(386.24,18.88)
(388.625,17.5)(391.04,16.08)(393.485,14.62)(395.96,13.12)(398.465,11.58)
(401,10)
\path (278,54)(280.2,55.005)(282.4,55.92)(284.6,56.745)(286.8,57.48)
(289,58.125)(291.2,58.68)(293.4,59.145)(295.6,59.52)(297.8,59.805)
(300,60)(302.2,60.105)(304.4,60.12)(306.6,60.045)(308.8,59.88)
(311,59.625)(313.2,59.28)(315.4,58.845)(317.6,58.32)(319.8,57.705)
(322,57)
\path (315.371,61.478)(322,57)(314.043,56.169)
\put(290,68){{\xiipt\rm {$X(t)$}}}
\put(409,2){{\xiipt\rm {$X_1=X(T)$}}}
\put(176,-1){{\xiipt\rm {$X_0$}}}
\end{picture}
\end{center}
\caption{problema da controlabilidade.}
\label{probcontrolabilidade}
\end{figure}

(ii) Assegurada a controlabilidade (quest\~{a}o anterior),
existir\'{a} um controlo que \emph{minimiza o tempo
de percurso de $X_0$ at\'{e} $X_1$}?
(Figura~\ref{probcontroloptimo}) Temos ent\~{a}o um problema
de controlo \'{o}ptimo (de tempo m\'{\i}nimo).

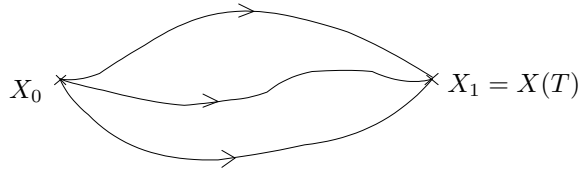
\begin{figure}[!htb]
\begin{center}
\unitlength=.25mm
\makeatletter
\def\shade{\@ifnextchar[{\shade@special}{\@killglue\special{sh}\ignorespaces}}
\def\shade@special[#1]{\@killglue\special{sh #1}\ignorespaces}
\makeatother
\begin{picture}(321,97)(176,-5)
\thinlines
\typeout{\space\space\space eepic-ture exported by 'qfig'.}
\font\FonttenBI=cmbxti10\relax
\font\FonttwlBI=cmbxti10 scaled \magstep1\relax
\path (201,49)(206,44)
\path (200,44)(206,49)
\path (398,49)(404,44)
\path (397,44)(404,50)
\path (203,47)(204.965,46.76)(206.96,46.64)(208.985,46.64)(211.04,46.76)
(213.125,47)(215.24,47.36)(217.385,47.84)(219.56,48.44)(221.765,49.16)
(224,50)(226.355,51.59)(228.72,53.16)(231.095,54.71)(233.48,56.24)
(235.875,57.75)(238.28,59.24)(240.695,60.71)(243.12,62.16)(245.555,63.59)
(248,65)(250.545,66.66)(253.08,68.24)(255.605,69.74)(258.12,71.16)
(260.625,72.5)(263.12,73.76)(265.605,74.94)(268.08,76.04)(270.545,77.06)
(273,78)(275.31,78.815)(277.64,79.56)(279.99,80.235)(282.36,80.84)
(284.75,81.375)(287.16,81.84)(289.59,82.235)(292.04,82.56)(294.51,82.815)
(297,83)(299.285,83.115)(301.64,83.16)(304.065,83.135)(306.56,83.04)
(309.125,82.875)(311.76,82.64)(314.465,82.335)(317.24,81.96)(320.085,81.515)
(323,81)(326.84,80.145)(330.56,79.28)(334.16,78.405)(337.64,77.52)
(341,76.625)(344.24,75.72)(347.36,74.805)(350.36,73.88)(353.24,72.945)
(356,72)(357.965,71.18)(359.96,70.32)(361.985,69.42)(364.04,68.48)
(366.125,67.5)(368.24,66.48)(370.385,65.42)(372.56,64.32)(374.765,63.18)
(377,62)(379.265,60.78)(381.56,59.52)(383.885,58.22)(386.24,56.88)
(388.625,55.5)(391.04,54.08)(393.485,52.62)(395.96,51.12)(398.465,49.58)
(401,48)
\put(409,40){{\xiipt\rm {$X_1=X(T)$}}}
\put(176,37){{\xiipt\rm {$X_0$}}}
\path (203,47)(204.4,46.49)(206,45.96)(207.8,45.41)(209.8,44.84)
(212,44.25)(214.4,43.64)(217,43.01)(219.8,42.36)(222.8,41.69)
(226,41)(231.11,39.705)(236.04,38.52)(240.79,37.445)(245.36,36.48)
(249.75,35.625)(253.96,34.88)(257.99,34.245)(261.84,33.72)(265.51,33.305)
(269,33)(271.77,33.255)(274.48,33.52)(277.13,33.795)(279.72,34.08)
(282.25,34.375)(284.72,34.68)(287.13,34.995)(289.48,35.32)(291.77,35.655)
(294,36)(295.675,36.085)(297.4,36.24)(299.175,36.465)(301,36.76)
(302.875,37.125)(304.8,37.56)(306.775,38.065)(308.8,38.64)(310.875,39.285)
(313,40)(315.04,41.595)(317.16,43.08)(319.36,44.455)(321.64,45.72)
(324,46.875)(326.44,47.92)(328.96,48.855)(331.56,49.68)(334.24,50.395)
(337,51)(340.875,51.225)(344.6,51.4)(348.175,51.525)(351.6,51.6)
(354.875,51.625)(358,51.6)(360.975,51.525)(363.8,51.4)(366.475,51.225)
(369,51)(370.79,50.23)(372.56,49.52)(374.31,48.87)(376.04,48.28)
(377.75,47.75)(379.44,47.28)(381.11,46.87)(382.76,46.52)(384.39,46.23)
(386,46)(387.59,45.83)(389.16,45.72)(390.71,45.67)(392.24,45.68)
(393.75,45.75)(395.24,45.88)(396.71,46.07)(398.16,46.32)(399.59,46.63)
(401,47)
\path (203,47)(203.825,45.1)(204.8,43.2)(205.925,41.3)(207.2,39.4)
(208.625,37.5)(210.2,35.6)(211.925,33.7)(213.8,31.8)(215.825,29.9)
(218,28)(220.595,25.47)(223.28,23.08)(226.055,20.83)(228.92,18.72)
(231.875,16.75)(234.92,14.92)(238.055,13.23)(241.28,11.68)(244.595,10.27)
(248,9)(251.585,7.915)(255.24,6.96)(258.965,6.135)(262.76,5.44)
(266.625,4.875)(270.56,4.44)(274.565,4.135)(278.64,3.96)(282.785,3.915)
(287,4)(291.96,4.665)(296.84,5.36)(301.64,6.085)(306.36,6.84)
(311,7.625)(315.56,8.44)(320.04,9.285)(324.44,10.16)(328.76,11.065)
(333,12)(337.16,12.515)(341.24,13.16)(345.24,13.935)(349.16,14.84)
(353,15.875)(356.76,17.04)(360.44,18.335)(364.04,19.76)(367.56,21.315)
(371,23)(374.36,24.815)(377.64,26.76)(380.84,28.835)(383.96,31.04)
(387,33.375)(389.96,35.84)(392.84,38.435)(395.64,41.16)(398.36,44.015)
(401,47)
\path (298,87)(306,83)(298,80)
\path (289,9)(296,5)(289,0)
\path (279,39)(287,35)(279,32)
\end{picture}
\end{center}
\caption{problema do tempo m\'{\i}nimo.}
\label{probcontroloptimo}
\end{figure}

Os teoremas que se seguem respondem a estas quest\~{o}es.
As respectivas demonstra\c{c}\~{o}es s\~{a}o bem conhecidas
e podem facilmente ser encontradas na literatura
(\textrm{vide}, \textrm{e.g.}, \cite{Lee_Markus,MR0638591,Trelat_book1}).


\subsection{Conjunto acess\'{\i}vel}

Considerando o sistema linear de controlo \eqref{contsystlin}
come\c{c}amos por introduzir um conjunto de grande import\^{a}ncia:
\emph{o conjunto acess\'{\i}vel}.

\begin{definition}
O conjunto dos pontos acess\'{\i}veis a partir de $X_0$
em tempo $T > 0$ \'{e} denotado e definido por
\begin{equation*}
\begin{split}
A(X_0, T) = \{ X_1 \in \R^n \, | \, &\exists u
\in L^{1}([0, T],I),\\
&\exists X : \R \to \R^n \, \in AC \text{ com } X(0) = X_0, \\
&\forall t \in [0, T] \, \, \dot{X}(t) = A X(t) + B u(t), \, \, X(T) = X_1 \} \, .
\end{split}
\end{equation*}
\end{definition}

Por palavras, $A(X_0, T)$ \'{e} o conjunto das extremidades das solu\c{c}\~{o}es
de \eqref{contsystlin} em tempo $T$, quando fazemos variar
o controlo $u$ (Figura \ref{imgacceset}).

\begin{figure}[!htb]
\begin{center}
\unitlength=.25mm
\makeatletter
\def\shade{\@ifnextchar[{\shade@special}{\@killglue\special{sh}\ignorespaces}}
\def\shade@special[#1]{\@killglue\special{sh #1}\ignorespaces}
\makeatother
\begin{picture}(277,170)(210,-5)
\thinlines
\typeout{\space\space\space eepic-ture exported by 'qfig'.}
\font\FonttenBI=cmbxti10\relax
\font\FonttwlBI=cmbxti10 scaled \magstep1\relax
\path (296,77)(303,71)
\path (296,71)(303,78)
\path (299,74)(300.59,74.31)(302.16,74.64)(303.71,74.99)(305.24,75.36)
(306.75,75.75)(308.24,76.16)(309.71,76.59)(311.16,77.04)(312.59,77.51)
(314,78)(315.165,77.925)(316.36,78)(317.585,78.225)(318.84,78.6)
(320.125,79.125)(321.44,79.8)(322.785,80.625)(324.16,81.6)(325.565,82.725)
(327,84)(328.87,85.785)(330.68,87.64)(332.43,89.565)(334.12,91.56)
(335.75,93.625)(337.32,95.76)(338.83,97.965)(340.28,100.24)(341.67,102.585)
(343,105)(343.55,108.205)(344.2,111.32)(344.95,114.345)(345.8,117.28)
(346.75,120.125)(347.8,122.88)(348.95,125.545)(350.2,128.12)(351.55,130.605)
(353,133)(354.91,135.575)(356.84,138)(358.79,140.275)(360.76,142.4)
(362.75,144.375)(364.76,146.2)(366.79,147.875)(368.84,149.4)(370.91,150.775)
(373,152)(375.11,153.075)(377.24,154)(379.39,154.775)(381.56,155.4)
(383.75,155.875)(385.96,156.2)(388.19,156.375)(390.44,156.4)(392.71,156.275)
(395,156)
\path (299,75)(305.56,73.575)(311.84,72.4)(317.84,71.475)(323.56,70.8)
(329,70.375)(334.16,70.2)(339.04,70.275)(343.64,70.6)(347.96,71.175)
(352,72)(353.69,74.74)(355.56,77.36)(357.61,79.86)(359.84,82.24)
(362.25,84.5)(364.84,86.64)(367.61,88.66)(370.56,90.56)(373.69,92.34)
(377,94)(382.245,95.675)(387.28,97.2)(392.105,98.575)(396.72,99.8)
(401.125,100.875)(405.32,101.8)(409.305,102.575)(413.08,103.2)(416.645,103.675)
(420,104)(423.145,104.175)(426.08,104.2)(428.805,104.075)(431.32,103.8)
(433.625,103.375)(435.72,102.8)(437.605,102.075)(439.28,101.2)(440.745,100.175)
(442,99)
\path (298,75)(298.35,72.785)(298.8,70.64)(299.35,68.565)(300,66.56)
(300.75,64.625)(301.6,62.76)(302.55,60.965)(303.6,59.24)(304.75,57.585)
(306,56)(308.07,54.755)(310.08,53.52)(312.03,52.295)(313.92,51.08)
(315.75,49.875)(317.52,48.68)(319.23,47.495)(320.88,46.32)(322.47,45.155)
(324,44)(325.29,43.395)(326.56,42.68)(327.81,41.855)(329.04,40.92)
(330.25,39.875)(331.44,38.72)(332.61,37.455)(333.76,36.08)(334.89,34.595)
(336,33)(337.09,31.295)(338.16,29.48)(339.21,27.555)(340.24,25.52)
(341.25,23.375)(342.24,21.12)(343.21,18.755)(344.16,16.28)(345.09,13.695)
(346,11)
\path (298,75)(296.01,77.105)(294.04,79.12)(292.09,81.045)(290.16,82.88)
(288.25,84.625)(286.36,86.28)(284.49,87.845)(282.64,89.32)(280.81,90.705)
(279,92)(277.3,92.98)(275.6,93.92)(273.9,94.82)(272.2,95.68)
(270.5,96.5)(268.8,97.28)(267.1,98.02)(265.4,98.72)(263.7,99.38)
(262,100)(260.885,100.715)(259.64,101.36)(258.265,101.935)(256.76,102.44)
(255.125,102.875)(253.36,103.24)(251.465,103.535)(249.44,103.76)(247.285,103.915)
(245,104)(242.585,104.015)(240.04,103.96)(237.365,103.835)(234.56,103.64)
(231.625,103.375)(228.56,103.04)(225.365,102.635)(222.04,102.16)(218.585,101.615)
(215,101)
\path (299,75)(294.86,74.025)(290.84,73)(286.94,71.925)(283.16,70.8)
(279.5,69.625)(275.96,68.4)(272.54,67.125)(269.24,65.8)(266.06,64.425)
(263,63)(259.52,61.525)(256.28,60)(253.28,58.425)(250.52,56.8)
(248,55.125)(245.72,53.4)(243.68,51.625)(241.88,49.8)(240.32,47.925)
(239,46)(238.415,43.665)(237.96,41.36)(237.635,39.085)(237.44,36.84)
(237.375,34.625)(237.44,32.44)(237.635,30.285)(237.96,28.16)(238.415,26.065)
(239,24)(239.715,21.965)(240.56,19.96)(241.535,17.985)(242.64,16.04)
(243.875,14.125)(245.24,12.24)(246.735,10.385)(248.36,8.56)(250.115,6.765)
(252,5)
\put(290,83){{\xiipt\rm {$X_0$}}}
\path (252,5)(257.65,3.78)(263.2,2.72)(268.65,1.82)(274,1.08)
(279.25,.5)(284.4,.08)(289.45,-.18)(294.4,-.28)(299.25,-.22)
(304,0)(307.48,.47)(311.12,1.08)(314.92,1.83)(318.88,2.72)
(323,3.75)(327.28,4.92)(331.72,6.23)(336.32,7.68)(341.08,9.27)
(346,11)(353.51,13.68)(360.64,16.32)(367.39,18.92)(373.76,21.48)
(379.75,24)(385.36,26.48)(390.59,28.92)(395.44,31.32)(399.91,33.68)
(404,36)(406.09,37.2)(408.16,38.6)(410.21,40.2)(412.24,42)
(414.25,44)(416.24,46.2)(418.21,48.6)(420.16,51.2)(422.09,54)
(424,57)(427.375,61.28)(430.4,65.52)(433.075,69.72)(435.4,73.88)
(437.375,78)(439,82.08)(440.275,86.12)(441.2,90.12)(441.775,94.08)
(442,98)(440.885,102.42)(439.64,106.68)(438.265,110.78)(436.76,114.72)
(435.125,118.5)(433.36,122.12)(431.465,125.58)(429.44,128.88)(427.285,132.02)
(425,135)(422.045,137.865)(419.08,140.56)(416.105,143.085)(413.12,145.44)
(410.125,147.625)(407.12,149.64)(404.105,151.485)(401.08,153.16)(398.045,154.665)
(395,156)(393.61,156.85)(391.84,157.6)(389.69,158.25)(387.16,158.8)
(384.25,159.25)(380.96,159.6)(377.29,159.85)(373.24,160)(368.81,160.05)
(364,160)(356.515,160.165)(349.16,160.16)(341.935,159.985)(334.84,159.64)
(327.875,159.125)(321.04,158.44)(314.335,157.585)(307.76,156.56)(301.315,155.365)
(295,154)(288.005,152.015)(281.32,149.96)(274.945,147.835)(268.88,145.64)
(263.125,143.375)(257.68,141.04)(252.545,138.635)(247.72,136.16)(243.205,133.615)
(239,131)(235.51,128.315)(232.24,125.56)(229.19,122.735)(226.36,119.84)
(223.75,116.875)(221.36,113.84)(219.19,110.735)(217.24,107.56)(215.51,104.315)
(214,101)(212.935,97.255)(212.04,93.52)(211.315,89.795)(210.76,86.08)
(210.375,82.375)(210.16,78.68)(210.115,74.995)(210.24,71.32)(210.535,67.655)
(211,64)(211.86,59.77)(212.84,55.68)(213.94,51.73)(215.16,47.92)
(216.5,44.25)(217.96,40.72)(219.54,37.33)(221.24,34.08)(223.06,30.97)
(225,28)(227.06,25.17)(229.24,22.48)(231.54,19.93)(233.96,17.52)
(236.5,15.25)(239.16,13.12)(241.94,11.13)(244.84,9.28)(247.86,7.57)
(251,6)
\put(399,19){{\xiipt\rm {$A(X_0,T)$}}}
\end{picture}
\end{center}
\caption{conjunto acess\'{\i}vel.}
\label{imgacceset}
\end{figure}
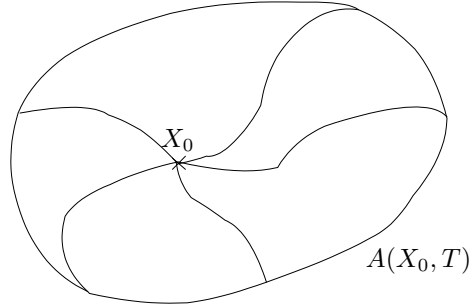

\begin{theorem}
\label{thm:1}
Sejam $T > 0$, $I$ compacto e $X_0 \in \R^n$. Ent\~{a}o para todo o
$t \in [0, T]$, $A(X_0, t)$ \'{e} compacto, convexo e varia continuamente com $t$ em $[0, T]$.
\end{theorem}

A solu\c{c}\~{a}o de
\begin{equation}
\label{eq:SL0}
\begin{cases}
\dot{X} = AX + Bu\\
X(0) = X_0
\end{cases}
\end{equation}
\'{e}
$$X(t) = e^{tA} + e^{tA} \int_0^t e^{-sA} Bu(s) ds \, .$$
Constatamos que se $X_0 = 0$, \textrm{i.e.}, se partirmos da origem,
ent\~{a}o a express\~{a}o de $X(t)$ \'{e} simplificada:
$X(t) = e^{tA} \int_{0}^{t} e^{-s A} Bu(s) ds$ \'{e} linear em $u$.
Esta observa\c{c}\~{a}o leva-nos \`{a} seguinte proposi\c{c}\~{a}o.

\begin{proposition}
Suponhamos que $X_0 = 0$ e $I = \R$. Ent\~{a}o,
\begin{itemize}
\item[1.] $\forall \, T > 0 \, \, \, A(0, T)$
\'{e} um sub-espa\c{c}o vectorial de $\R^n$. Al\'{e}m disso,
\item[2.] $0 < T_1 < T_2 \Rightarrow A(0, T_1) \subset A (0, T_2)$.
\end{itemize}
\end{proposition}

\begin{definition}
O conjunto $A(0) = \cup_{t \geq 0} A(0, T)$ \'{e} o conjunto
dos pontos acess\'{\i}veis (num tempo qualquer) a partir da origem.
\end{definition}

\begin{corollary}
O conjunto $A(0)$ \'{e} um sub-espa\c{c}o vectorial de $\R^n$.
\end{corollary}


\subsection{Controlabilidade}

O sistema de controlo $\dot{X} = AX + Bu$ diz-se control\'{a}vel
se para todo o $X_0, X_1 \in \R^n$ existe um controlo $u$
tal que a traject\'{o}ria associada une $X_0$ a $X_1$ em tempo finito $T$
(Figura~\ref{controlabilidade}). De modo mais formal temos:

\begin{definition}
O sistema de controlo $\dot{X} = AX + Bu$ diz-se control\'{a}vel se
$$\forall \, X_0, X_1 \in \R^n \hspace{0.3cm} \exists T> 0
\hspace{0.3cm} \exists u : [0, T] \to I \, \in L^1$$
\begin{equation*}
\exists X : [0, T] \to \R^n \, \, | \, \,
\begin{cases}
\dot{X} = AX + Bu \, ,\\
X(0) = X_0 \, ,\\
X(T) = X_1 \, .
\end{cases}
\end{equation*}
\end{definition}

\begin{figure}[!htb]
\begin{center}
\unitlength=.25mm
\makeatletter
\def\shade{\@ifnextchar[{\shade@special}{\@killglue\special{sh}\ignorespaces}}
\def\shade@special[#1]{\@killglue\special{sh #1}\ignorespaces}
\makeatother
\begin{picture}(292,40)(165,-5)
\thinlines
\typeout{\space\space\space eepic-ture exported by 'qfig'.}
\font\FonttenBI=cmbxti10\relax
\font\FonttwlBI=cmbxti10 scaled \magstep1\relax
\path (197,12)(202,7)
\path (196,7)(202,12)
\path (397,12)(402,7)
\path (395,7)(403,12)
\path (199,10)(200.875,9.785)(202.8,9.64)(204.775,9.565)(206.8,9.56)
(208.875,9.625)(211,9.76)(213.175,9.965)(215.4,10.24)(217.675,10.585)
(220,11)(223.005,11.935)(225.92,12.84)(228.745,13.715)(231.48,14.56)
(234.125,15.375)(236.68,16.16)(239.145,16.915)(241.52,17.64)(243.805,18.335)
(246,19)(247.7,19.635)(249.4,20.24)(251.1,20.815)(252.8,21.36)
(254.5,21.875)(256.2,22.36)(257.9,22.815)(259.6,23.24)(261.3,23.635)
(263,24)(263.935,24.47)(265.04,24.88)(266.315,25.23)(267.76,25.52)
(269.375,25.75)(271.16,25.92)(273.115,26.03)(275.24,26.08)(277.535,26.07)
(280,26)(283.49,25.78)(286.96,25.52)(290.41,25.22)(293.84,24.88)
(297.25,24.5)(300.64,24.08)(304.01,23.62)(307.36,23.12)(310.69,22.58)
(314,22)(317.38,21.47)(320.72,20.88)(324.02,20.23)(327.28,19.52)
(330.5,18.75)(333.68,17.92)(336.82,17.03)(339.92,16.08)(342.98,15.07)
(346,14)(348.935,11.52)(351.84,9.28)(354.715,7.28)(357.56,5.52)
(360.375,4)(363.16,2.72)(365.915,1.68)(368.64,.88)(371.335,.32)
(374,0)(376.635,-.08)(379.24,.08)(381.815,.48)(384.36,1.12)
(386.875,2)(389.36,3.12)(391.815,4.48)(394.24,6.08)(396.635,7.92)
(399,10)
\path (272,30)(278,26)(272,23)
\put(409,4){{\xiipt\rm {$X_1$}}}
\put(165,1){{\xiipt\rm {$X_0$}}}
\end{picture}
\end{center}
\caption{controlabilidade.}
\label{controlabilidade}
\end{figure}

O teorema seguinte d\'{a}-nos uma condi\c{c}\~{a}o necess\'{a}ria
e suficiente de controlabilidade
chamada \emph{condi\c{c}\~{a}o de Kalman}.

\begin{theorem}[Condi\c{c}\~{a}o de Kalman]
\label{condKalman}
O sistema $\dot{X} = AX + Bu$ \'{e} control\'{a}vel
se e somente se a matriz $C = (B | AB | \cdots | A^{n-1}B)$
tiver caracter\'{\i}stica completa
(\textrm{i.e.}, $rank(C) = n$).
\end{theorem}


\subsection{Princ\'{\i}pio do M\'{a}ximo de Pontryagin
para o problema de tempo m\'{\i}nimo}

Come\c{c}amos por formalizar, com a ajuda do conjunto
acess\'{\i}vel $A(X_0, t)$, a no\c{c}\~{a}o de tempo m\'{\i}nimo.

Sejam $X_0, X_1 \in \R^n$. Suponhamos que $X_1$ \'{e} acess\'{\i}vel
a partir de $X_0$, \textrm{i.e.}, suponhamos que existe pelo menos
uma traject\'{o}ria unindo $X_0$ a $X_1$. De entre todas as traject\'{o}rias
que unem $X_0$ a $X_1$ gostar\'{\i}amos de caracterizar aquela
que o faz em tempo m\'{\i}nimo $T$ (Figura~\ref{tempominimo}).

\begin{figure}[!htb]
\begin{center}
\unitlength=.25mm
\makeatletter
\def\shade{\@ifnextchar[{\shade@special}{\@killglue\special{sh}\ignorespaces}}
\def\shade@special[#1]{\@killglue\special{sh #1}\ignorespaces}
\makeatother
\begin{picture}(321,97)(176,-5)
\thinlines
\typeout{\space\space\space eepic-ture exported by 'qfig'.}
\font\FonttenBI=cmbxti10\relax
\font\FonttwlBI=cmbxti10 scaled \magstep1\relax
\path (201,49)(206,44)
\path (200,44)(206,49)
\path (398,49)(404,44)
\path (397,44)(404,50)
\path (203,47)(204.965,46.76)(206.96,46.64)(208.985,46.64)(211.04,46.76)
(213.125,47)(215.24,47.36)(217.385,47.84)(219.56,48.44)(221.765,49.16)
(224,50)(226.355,51.59)(228.72,53.16)(231.095,54.71)(233.48,56.24)
(235.875,57.75)(238.28,59.24)(240.695,60.71)(243.12,62.16)(245.555,63.59)
(248,65)(250.545,66.66)(253.08,68.24)(255.605,69.74)(258.12,71.16)
(260.625,72.5)(263.12,73.76)(265.605,74.94)(268.08,76.04)(270.545,77.06)
(273,78)(275.31,78.815)(277.64,79.56)(279.99,80.235)(282.36,80.84)
(284.75,81.375)(287.16,81.84)(289.59,82.235)(292.04,82.56)(294.51,82.815)
(297,83)(299.285,83.115)(301.64,83.16)(304.065,83.135)(306.56,83.04)
(309.125,82.875)(311.76,82.64)(314.465,82.335)(317.24,81.96)(320.085,81.515)
(323,81)(326.84,80.145)(330.56,79.28)(334.16,78.405)(337.64,77.52)
(341,76.625)(344.24,75.72)(347.36,74.805)(350.36,73.88)(353.24,72.945)
(356,72)(357.965,71.18)(359.96,70.32)(361.985,69.42)(364.04,68.48)
(366.125,67.5)(368.24,66.48)(370.385,65.42)(372.56,64.32)(374.765,63.18)
(377,62)(379.265,60.78)(381.56,59.52)(383.885,58.22)(386.24,56.88)
(388.625,55.5)(391.04,54.08)(393.485,52.62)(395.96,51.12)(398.465,49.58)
(401,48)
\put(409,40){{\xiipt\rm {$X_1=X(T)$}}}
\put(176,37){{\xiipt\rm {$X_0$}}}
\path (203,47)(204.4,46.49)(206,45.96)(207.8,45.41)(209.8,44.84)
(212,44.25)(214.4,43.64)(217,43.01)(219.8,42.36)(222.8,41.69)
(226,41)(231.11,39.705)(236.04,38.52)(240.79,37.445)(245.36,36.48)
(249.75,35.625)(253.96,34.88)(257.99,34.245)(261.84,33.72)(265.51,33.305)
(269,33)(271.77,33.255)(274.48,33.52)(277.13,33.795)(279.72,34.08)
(282.25,34.375)(284.72,34.68)(287.13,34.995)(289.48,35.32)(291.77,35.655)
(294,36)(295.675,36.085)(297.4,36.24)(299.175,36.465)(301,36.76)
(302.875,37.125)(304.8,37.56)(306.775,38.065)(308.8,38.64)(310.875,39.285)
(313,40)(315.04,41.595)(317.16,43.08)(319.36,44.455)(321.64,45.72)
(324,46.875)(326.44,47.92)(328.96,48.855)(331.56,49.68)(334.24,50.395)
(337,51)(340.875,51.225)(344.6,51.4)(348.175,51.525)(351.6,51.6)
(354.875,51.625)(358,51.6)(360.975,51.525)(363.8,51.4)(366.475,51.225)
(369,51)(370.79,50.23)(372.56,49.52)(374.31,48.87)(376.04,48.28)
(377.75,47.75)(379.44,47.28)(381.11,46.87)(382.76,46.52)(384.39,46.23)
(386,46)(387.59,45.83)(389.16,45.72)(390.71,45.67)(392.24,45.68)
(393.75,45.75)(395.24,45.88)(396.71,46.07)(398.16,46.32)(399.59,46.63)
(401,47)
\path (203,47)(203.825,45.1)(204.8,43.2)(205.925,41.3)(207.2,39.4)
(208.625,37.5)(210.2,35.6)(211.925,33.7)(213.8,31.8)(215.825,29.9)
(218,28)(220.595,25.47)(223.28,23.08)(226.055,20.83)(228.92,18.72)
(231.875,16.75)(234.92,14.92)(238.055,13.23)(241.28,11.68)(244.595,10.27)
(248,9)(251.585,7.915)(255.24,6.96)(258.965,6.135)(262.76,5.44)
(266.625,4.875)(270.56,4.44)(274.565,4.135)(278.64,3.96)(282.785,3.915)
(287,4)(291.96,4.665)(296.84,5.36)(301.64,6.085)(306.36,6.84)
(311,7.625)(315.56,8.44)(320.04,9.285)(324.44,10.16)(328.76,11.065)
(333,12)(337.16,12.515)(341.24,13.16)(345.24,13.935)(349.16,14.84)
(353,15.875)(356.76,17.04)(360.44,18.335)(364.04,19.76)(367.56,21.315)
(371,23)(374.36,24.815)(377.64,26.76)(380.84,28.835)(383.96,31.04)
(387,33.375)(389.96,35.84)(392.84,38.435)(395.64,41.16)(398.36,44.015)
(401,47)
\path (298,87)(306,83)(298,80)
\path (289,9)(296,5)(289,0)
\path (279,39)(287,35)(279,32)
\end{picture}
\end{center}
\caption{qual a traject\'{o}ria $X$ para a qual $T$ \'{e} m\'{\i}nimo?}
\label{tempominimo}
\end{figure}
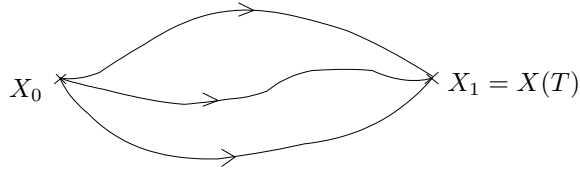

Se $T$ for o tempo m\'{\i}nimo, ent\~{a}o para todo o $t < T$, $X_1 \not \in A(X_0, T)$
(com efeito, se assim n\~{a}o fosse $X_1$ seria acess\'{\i}vel
a partir de $X_0$ num tempo inferior a $T$
e $T$ n\~{a}o seria o tempo m\'{\i}nimo). Consequentemente,
\begin{equation}
\label{eq:inf:T}
T = \inf \{ t> 0 \, | \, X_1 \in A(X_0, t) \} \, .
\end{equation}
O valor de $T$ est\'{a} bem definido pois, a partir do Teorema~\ref{thm:1},
$A(X_0, t)$ varia continuamente com $t$, logo $\{ t > 0 \, | \, X_1 \in A(X_0, t) \}$
\'{e} fechado em $\R$. Em particular o \'{\i}nfimo em \eqref{eq:inf:T} \'{e} m\'{\i}nimo.
O tempo $t=T$ \'{e} o primeiro instante para o qual $A(X_0, t)$
cont\'{e}m $X_1$ (Figura~\ref{tempominimo2}).

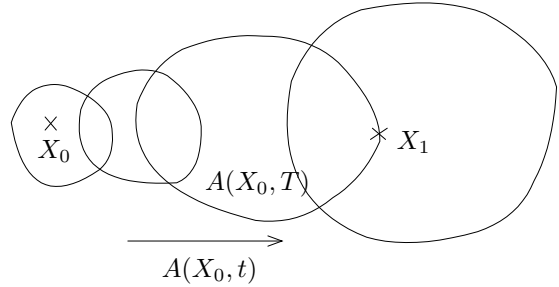
\begin{figure}[!htb]
\begin{center}
\unitlength=.25mm
\makeatletter
\def\shade{\@ifnextchar[{\shade@special}{\@killglue\special{sh}\ignorespaces}}
\def\shade@special[#1]{\@killglue\special{sh #1}\ignorespaces}
\makeatother
\begin{picture}(290,155)(167,-5)
\thinlines
\typeout{\space\space\space eepic-ture exported by 'qfig'.}
\font\FonttenBI=cmbxti10\relax
\font\FonttwlBI=cmbxti10 scaled \magstep1\relax
\path (185,84)(191,77)
\path (185,77)(191,84)
\put(181,62){{\xiipt\rm {$X_0$}}}
\path (167,81)(167.985,83.61)(169.04,86.04)(170.165,88.29)(171.36,90.36)
(172.625,92.25)(173.96,93.96)(175.365,95.49)(176.84,96.84)(178.385,98.01)
(180,99)(182,99.81)(184,100.44)(186,100.89)(188,101.16)
(190,101.25)(192,101.16)(194,100.89)(196,100.44)(198,99.81)
(200,99)(203.17,97.515)(206.08,95.96)(208.73,94.335)(211.12,92.64)
(213.25,90.875)(215.12,89.04)(216.73,87.135)(218.08,85.16)(219.17,83.115)
(220,81)(220.525,77.69)(220.8,74.56)(220.825,71.61)(220.6,68.84)
(220.125,66.25)(219.4,63.84)(218.425,61.61)(217.2,59.56)(215.725,57.69)
(214,56)(210.225,53.545)(206.6,51.48)(203.125,49.805)(199.8,48.52)
(196.625,47.625)(193.6,47.12)(190.725,47.005)(188,47.28)(185.425,47.945)
(183,49)(180.725,50.445)(178.6,52.28)(176.625,54.505)(174.8,57.12)
(173.125,60.125)(171.6,63.52)(170.225,67.305)(169,71.48)(167.925,76.045)
(167,81)
\path (204,88)(204.97,90.42)(206.08,92.68)(207.33,94.78)(208.72,96.72)
(210.25,98.5)(211.92,100.12)(213.73,101.58)(215.68,102.88)(217.77,104.02)
(220,105)(223.54,106.36)(226.96,107.44)(230.26,108.24)(233.44,108.76)
(236.5,109)(239.44,108.96)(242.26,108.64)(244.96,108.04)(247.54,107.16)
(250,106)(253.195,103.435)(256.08,100.84)(258.655,98.215)(260.92,95.56)
(262.875,92.875)(264.52,90.16)(265.855,87.415)(266.88,84.64)(267.595,81.835)
(268,79)(268.005,74.29)(267.72,69.96)(267.145,66.01)(266.28,62.44)
(265.125,59.25)(263.68,56.44)(261.945,54.01)(259.92,51.96)(257.605,50.29)
(255,49)(249.315,48.765)(243.96,48.76)(238.935,48.985)(234.24,49.44)
(229.875,50.125)(225.84,51.04)(222.135,52.185)(218.76,53.56)(215.715,55.165)
(213,57)(210.615,59.065)(208.56,61.36)(206.835,63.885)(205.44,66.64)
(204.375,69.625)(203.64,72.84)(203.235,76.285)(203.16,79.96)(203.415,83.865)
(204,88)
\path (234,91)(235.545,94.89)(237.28,98.56)(239.205,102.01)(241.32,105.24)
(243.625,108.25)(246.12,111.04)(248.805,113.61)(251.68,115.96)(254.745,118.09)
(258,120)(262.435,121.645)(266.84,123.08)(271.215,124.305)(275.56,125.32)
(279.875,126.125)(284.16,126.72)(288.415,127.105)(292.64,127.28)(296.835,127.245)
(301,127)(305.81,126.86)(310.44,126.44)(314.89,125.74)(319.16,124.76)
(323.25,123.5)(327.16,121.96)(330.89,120.14)(334.44,118.04)(337.81,115.66)
(341,113)(345.315,108.26)(349.16,103.64)(352.535,99.14)(355.44,94.76)
(357.875,90.5)(359.84,86.36)(361.335,82.34)(362.36,78.44)(362.915,74.66)
(363,71)(361.22,67.19)(359.28,63.56)(357.18,60.11)(354.92,56.84)
(352.5,53.75)(349.92,50.84)(347.18,48.11)(344.28,45.56)(341.22,43.19)
(338,41)(334.35,38.18)(330.6,35.72)(326.75,33.62)(322.8,31.88)
(318.75,30.5)(314.6,29.48)(310.35,28.82)(306,28.52)(301.55,28.58)
(297,29)(290.145,30.77)(283.68,32.68)(277.605,34.73)(271.92,36.92)
(266.625,39.25)(261.72,41.72)(257.205,44.33)(253.08,47.08)(249.345,49.97)
(246,53)(243.045,56.17)(240.48,59.48)(238.305,62.93)(236.52,66.52)
(235.125,70.25)(234.12,74.12)(233.505,78.13)(233.28,82.28)(233.445,86.57)
(234,91)
\path (316,103)(318.395,108.6)(321.08,113.8)(324.055,118.6)(327.32,123)
(330.875,127)(334.72,130.6)(338.855,133.8)(343.28,136.6)(347.995,139)
(353,141)(360.86,142.465)(368.44,143.56)(375.74,144.285)(382.76,144.64)
(389.5,144.625)(395.96,144.24)(402.14,143.485)(408.04,142.36)(413.66,140.865)
(419,139)(425.05,135.595)(430.6,132.08)(435.65,128.455)(440.2,124.72)
(444.25,120.875)(447.8,116.92)(450.85,112.855)(453.4,108.68)(455.45,104.395)
(457,100)(456.61,93.965)(456.04,88.16)(455.29,82.585)(454.36,77.24)
(453.25,72.125)(451.96,67.24)(450.49,62.585)(448.84,58.16)(447.01,53.965)
(445,50)(443.305,46.22)(441.32,42.68)(439.045,39.38)(436.48,36.32)
(433.625,33.5)(430.48,30.92)(427.045,28.58)(423.32,26.48)(419.305,24.62)
(415,23)(408.065,21.17)(401.36,19.68)(394.885,18.53)(388.64,17.72)
(382.625,17.25)(376.84,17.12)(371.285,17.33)(365.96,17.88)(360.865,18.77)
(356,20)(350.96,22.155)(346.24,24.52)(341.84,27.095)(337.76,29.88)
(334,32.875)(330.56,36.08)(327.44,39.495)(324.64,43.12)(322.16,46.955)
(320,51)(318.16,55.255)(316.64,59.72)(315.44,64.395)(314.56,69.28)
(314,74.375)(313.76,79.68)(313.84,85.195)(314.24,90.92)(314.96,96.855)
(316,103)
\path (359,78)(366,72)
\path (358,72)(367,78)
\put(372,67){{\xiipt\rm {$X_1$}}}
\put(270,45){{\xiipt\rm {$A(X_0,T)$}}}
\path (229,18)(311,18)
\path (303.482,20.736)(311,18)(303.482,15.264)
\put(247,-1){{\xiipt\rm {$A(X_0,t)$}}}
\end{picture}
\end{center}
\caption{o tempo m\'{\i}nimo $T$ corresponde ao primeiro instante $t$
para o qual $A(X_0, t) \cap \{X_1\} \ne \emptyset$.}
\label{tempominimo2}
\end{figure}

Por outro lado, temos necessariamente:
\begin{equation*}
X_1 \in Fr \, A(X_0, T) \backslash int \, A(X_0, T) \, .
\end{equation*}
Com efeito, se $X_1$ pertencesse ao interior de $A(X_0, T)$,
ent\~{a}o para $t < T$ pr\'{o}ximo de $T$, $X_1$ pertenceria ainda a
$A(X_0, t)$ pois $A(X_0, t)$ varia continuamente com $t$.
Isto contradiz o facto de $T$ ser o tempo m\'{\i}nimo.
Estas observa\c{c}\~{o}es d\~{a}o uma vis\~{a}o geom\'{e}trica
\`{a} no\c{c}\~{a}o de tempo m\'{\i}nimo e conduzem-nos
\`{a} seguinte defini\c{c}\~{a}o:

\begin{definition}
Seja $u \in L^1([0, T], I)$. O controlo $u$ diz-se \'{o}ptimo
para o sistema \eqref{contsystlin} se a correspondente
traject\'{o}ria $X$ verifica $X(T) \in Fr \, A(X_0, T)$.
\end{definition}

Dizer que $u$ \'{e} \'{o}ptimo \'{e} dizer que a traject\'{o}ria associada
a $u$ une $X_0$ a $X_1$ em tempo m\'{\i}nimo. O objectivo
\'{e} ent\~{a}o o de determinar os controlos \'{o}ptimos. O teorema que se segue
d\'{a}-nos uma condi\c{c}\~{a}o necess\'{a}ria e suficiente de optimalidade.

\begin{theorem}[Princ\'{\i}pio do M\'{a}ximo de Pontryagin (caso linear)]
\label{principiomaxLinear}
Considere-se o sistema de controlo
\begin{equation*}
\begin{cases}
\dot{X} = AX + Bu \, ,\\
X(0) = X_0 \, .
\end{cases}
\end{equation*}
Seja $T > 0$. O controlo $u \in L^1([0, T],I = [-1, 1])$
\'{e} \'{o}ptimo se e somente se
$$u(t) = sinal \langle \eta(t), B \rangle$$
onde $\langle \cdot, \cdot \rangle$ \'{e} o produto
interno em $\mathbb{R}^n$ e $\eta(t) \in \R^n$
\'{e} solu\c{c}\~{a}o da equa\c{c}\~{a}o $\dot{\eta}^T = - \eta^T A$.
\end{theorem}

A condi\c{c}\~{a}o inicial $\eta(0)$ depende de $X_1$.
Como ela n\~{a}o \'{e} directamente conhecida, a utiliza\c{c}\~{a}o do
Teorema~\ref{principiomaxLinear} \'{e} maioritariamente indirecta.
Vejamos um exemplo.


\subsection{Exemplo: controlo \'{o}ptimo
de um oscilador harm\'{o}nico (caso linear)}

Consideremos uma massa pontual $m$
ligada a uma mola cujo movimento est\'{a} restrito
a um eixo $Ox$ (Figura~\ref{mola}).

\begin{figure}[!htb]
\begin{center}
\unitlength=.25mm
\makeatletter
\def\shade{\@ifnextchar[{\shade@special}{\@killglue\special{sh}\ignorespaces}}
\def\shade@special[#1]{\@killglue\special{sh #1}\ignorespaces}
\makeatother
\begin{picture}(239,109)(191,-5)
\thinlines
\typeout{\space\space\space eepic-ture exported by 'qfig'.}
\font\FonttenBI=cmbxti10\relax
\font\FonttwlBI=cmbxti10 scaled \magstep1\relax
\path (200,0)(200,80)
\path (200,75)(192,70)
\path (200,66)(192,61)
\path (199,47)(191,42)
\path (199,56)(191,51)
\path (200,38)(192,33)
\path (200,29)(192,24)
\path (200,21)(192,16)
\path (200,13)(192,8)
\path (200,40)(400,40)
\path (392.482,42.736)(400,40)(392.482,37.264)
\put(390,19){{\xiipt\rm {$x$}}}
\path (200,40)(212,51)(224,28)(236,52)(248,28)
(260,52)(272,28)(284,52)(296,28)(305,40)
(305,40)
\shade[0.7]\path (306,46)(306,33)
(327,33)(327,46)(306,46)
\put(311,15){{\xiipt\rm {$m$}}}
\path (275,79)(321,79)
\path (313.482,81.736)(321,79)(313.482,76.264)
\put(283,87){{\xiipt\rm {$\vec{\iota}$}}}
\put(203,24){{\xiipt\rm {$O$}}}
\end{picture}
\end{center}
\caption{sistema massa-mola.}
\label{mola}
\end{figure}
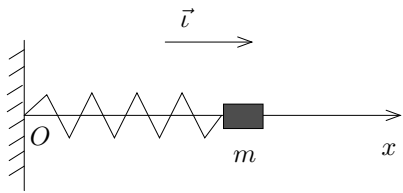
A massa pontual sai da origem por uma for\c{c}a que supomos igual a
$$-k_1(x-l) - k_2(x-l)^3$$ onde $l$ \'{e} o comprimento da mola em repouso.
Aplicamos a essa massa pontual uma for\c{c}a exterior horizontal $u(t) \vec{l}$.
A segunda Lei de Newton diz-nos que
a for\c{c}a resultante aplicada \'{e} directamente proporcional ao produto
entre a massa inercial e a acelera\c{c}\~{a}o adquirida pela mesma, ou seja
\begin{equation}
\label{eqmov}
m\ddot{x}(t) + k_1(x(t)-l) + k_2 (x(t) - l)^3 = u(t) \, .
\end{equation}
As leis b\'{a}sicas da F\'{\i}sica dizem-nos tamb\'{e}m que
todas as for\c{c}as s\~{a}o limitadas. Impomos a seguinte
\emph{restri\c{c}\~{a}o} \`{a} for\c{c}a exterior:
\begin{equation*}
|u(t)| \leq 1 \quad \forall \, t \, .
\end{equation*}
Isto significa que a for\c{c}a apenas pode tomar valores
no intervalo \emph{fechado} $[-1,1]$.
Suponhamos que a posi\c{c}\~{a}o e a velocidade iniciais do objecto s\~{a}o,
respectivamente, $x(0) = x_0$ e $\dot{x}(0) = y_0$.
O problema consiste em trazer, em \emph{tempo m\'{\i}nimo},
a massa pontual \`{a} posi\c{c}\~{a}o de equil\'{\i}brio $x = l$ por escolha
adequada da for\c{c}a externa $u(t)$ e tendo em conta a \emph{restri\c{c}\~{a}o}
$|u(t)| \leq 1$. A for\c{c}a $u$ \'{e} aqui o nosso \emph{controlo}.

\begin{problema}
Dadas as condi\c{c}\~{o}es iniciais $x(0) = x_0$ e $\dot{x}(0) = y_0$,
encontrar a fun\c{c}\~{a}o $u$ que permite transportar a massa
para a sua posi\c{c}\~{a}o de equil\'{\i}brio em tempo m\'{\i}nimo.
\end{problema}


\subsubsection{Modela\c{c}\~{a}o matem\'{a}tica}

Para simplificar a apresenta\c{c}\~{a}o, vamos supor
$m = 1 \, kg$, $k_1 = 1 \, N.m^{-1}$ e $l = 0 \, m$
(passamos a $l = 0$ por transla\c{c}\~{a}o).
A equa\c{c}\~{a}o de movimento \eqref{eqmov} \'{e} ent\~{a}o equivalente
ao sistema diferencial de controlo
\begin{equation}
\label{SC}
\begin{gathered}
\begin{cases}
\dot{x}(t) = y(t)\\
\dot{y}(t) = - x(t) - k_2 x(t)^3 + u(t)
\end{cases}\\
x(0) = x_0, \, \, \, \dot{x}(0) = y_0 \, .
\end{gathered}
\end{equation}
Escrevemos facilmente \eqref{SC} na nota\c{c}\~{a}o matricial
\begin{equation}
\label{SCFD}
\dot{X} = A X + f(X) + Bu \, , \hspace{0.3 cm} X(0) = X_0 \, ,
\end{equation}
tomando
\begin{equation*}
\begin{split}
A &= \left( \begin{array}{cc} 0 & 1\\
-1 & 0
\end{array} \right)\, , \, \,
B = \left( \begin{array}{c} 0\\
1 \end{array} \right)\, , \\
X &= \left( \begin{array}{c} x\\
y \end{array} \right)\, , \, \,
X_0 = \left( \begin{array}{c} x_0\\
y_0 \end{array} \right)\, , \, \,
f(X) = \left( \begin{array}{c} 0\\
-k_2 x^3 \end{array} \right) \, .
\end{split}
\end{equation*}
Tendo em mente que estamos na sec\c{c}\~{a}o de controlo linear
fixamos $k_2 = 0$, desprezando efeitos conservativos
n\~{a}o lineares (na Sec\c{c}\~{a}o~\ref{sec:CO:NL},
onde abordamos o controlo \'{o}ptimo n\~{a}o linear,
consideraremos o caso $k_2 \ne 0$).
Para $k_2 =0$ temos $f(X) \equiv 0$ e obtemos
o sistema de controlo \eqref{SCFD} na forma \eqref{eq:SL0}
(sistema de controlo linear).
Pretendemos responder a duas quest\~{o}es:
\begin{itemize}
\item[1.] Existir\'{a} sempre, para toda e qualquer condi\c{c}\~{a}o inicial
$x(0) = x_0$ e $\dot{x}(0) = y_0$,  uma for\c{c}a exterior horizontal
(um controlo) que permite transportar em tempo finito $T$ a massa
pontual para sua posi\c{c}\~{a}o de equil\'{\i}brio $x(T) = 0$ e $\dot{x}(T) = 0$?

\item[2.] Se a primeira pergunta for respondida positivamente,
qual a for\c{c}a (qual o controlo) que minimiza o tempo
de transporte da massa pontual \`{a} sua posi\c{c}\~{a}o de equil\'{\i}brio?
\end{itemize}


\subsubsection{Controlabilidade do sistema}

O nosso sistema escreve-se na forma
\begin{equation*}
\begin{cases}
\dot{X} = AX + B u\\
X(0) = X_0
\end{cases}
\end{equation*}
com $A= \left(\begin{array}{cc} 0 & 1\\ -1 & 0 \end{array}\right)$
e $B= \left( \begin{array}{c} 0\\ 1 \end{array}\right)$.
Temos ent\~{a}o
$$rank\left( B | AB \right) = rank\left(\begin{array}{cc} 0 & 1\\
1 & 0 \end{array} \right) = 2$$ e o Teorema~\ref{condKalman}
garante-nos que o sistema \'{e} control\'{a}vel (se $u(t) \in \R$).
Isto significa que existem controlos para os quais as traject\'{o}rias
associadas unem $X_0$ a $0$. Temos assim resposta afirmativa
\`{a} nossa primeira quest\~{a}o, admitindo que o sistema mant\'{e}m-se control\'{a}vel
com controlos que verificam a restri\c{c}\~{a}o $|u| \leq 1$
(o que ser\'{a} verificado \emph{a posteriori}).
Esta resposta \'{e} esperada em termos f\'{\i}sicos.
Se n\~{a}o aplicarmos uma for\c{c}a exterior, \textrm{i.e.},
se $u = 0$, a equa\c{c}\~{a}o do movimento \'{e} $\ddot{x} + x = 0$
e a massa pontual oscila sem nunca parar, nunca voltando
\`{a} sua posi\c{c}\~{a}o de equil\'{\i}brio em tempo finito. Por outro lado,
ao aplicarmos determinadas for\c{c}as exteriores, temos tend\^{e}ncia
a amortecer as oscila\c{c}\~{o}es. A teoria do controlo prev\^{e}
que conseguimos realmente parar a massa em tempo finito.


\subsubsection{Determina\c{c}\~{a}o do controlo \'{o}ptimo}

Sabemos que existem controlos que permitem conduzir o sistema de $X_0$ a $0$.
Agora queremos determinar, em concreto,
qual desses controlos o faz em tempo m\'{\i}nimo.
Para isso aplicamos o Teorema~\ref{principiomaxLinear}:
$$u(t) = sinal \, \langle \eta(t), B \rangle \, ,$$
onde $\eta(t) \in \R^2$ \'{e} solu\c{c}\~{a}o de $\dot{\eta}^T = - \eta^T A$.
Seja $\eta(t) = \left(\begin{array}{c} \eta_1(t)\\
\eta_2(t) \end{array} \right)$. Ent\~{a}o, $u(t) = sinal \,
\eta_2(t)$ e $\dot{\eta}_1 = \eta_2$, $\dot{\eta}_2 = - \eta_1$,
ou seja, $\ddot{\eta}_2 + \eta_2 = 0$.
Logo $\eta_2 (t) = \lambda \cos t + \mu \sin t $.
Consequentemente, o controlo \'{o}ptimo \'{e} seccionalmente constante
em intervalos de comprimento $\pi$ e toma valores alternadamente $\pm 1$.
\begin{itemize}
\item[$\bullet$] Se $u = -1$, obtemos o sistema diferencial
\begin{equation}
\label{sistdif-1}
\begin{cases}
\dot{x} = y \, ,\\
\dot{y} = -x -1 \, .
\end{cases}
\end{equation}

\item[$\bullet$] Se $u = + 1$, obtemos
\begin{equation}
\label{sistdif+1}
\begin{cases}
\dot{x} = y \, ,\\
\dot{y} = -x + 1 \, .
\end{cases}
\end{equation}
\end{itemize}
A traject\'{o}ria \'{o}ptima unindo $X_0$ a $0$ \'{e} constitu\'{\i}da por peda\c{c}os
de solu\c{c}\~{o}es de \eqref{sistdif-1} e \eqref{sistdif+1} concatenadas.
As solu\c{c}\~{o}es de \eqref{sistdif-1} e \eqref{sistdif+1} s\~{a}o
facilmente obtidas:
\begin{equation*}
\begin{split}
\dot{x} = y, \, \dot{y} = -x -1 \, &\Rightarrow \frac{d}{dx}((x+1)^2 + y^2) = 0\\
&\Rightarrow (x+1)^2 + y^2 = const = R^2
\end{split}
\end{equation*}
e conclu\'{\i}mos que as curvas solu\c{c}\~{o}es de \eqref{sistdif-1}
s\~{a}o c\'{\i}rculos centrados em $x=-1$ e $y=0$ de per\'{\i}odo $2\pi$
(com efeito, $x(t) = -1 + R \cos t$ e $y(t) = R \sin t$);
como solu\c{c}\~{o}es de \eqref{sistdif+1} obtemos
$x(t) = 1 + R \cos t$ e $y(t) = R \sin t$, \textrm{i.e.},
as solu\c{c}\~{o}es de \eqref{sistdif+1} s\~{a}o c\'{\i}rculos centrados
em $x=1$ e $y=0$ de per\'{\i}odo $2\pi$.

A traject\'{o}ria \'{o}ptima de $X_0$ at\'{e} $0$ segue alternadamente
um arco de c\'{\i}rculo centrado em $x=-1$ e $y=0$
e um arco de c\'{\i}rculo centrado em $x=1$ e $y=0$.
O estudo detalhado da traject\'{o}ria \'{o}ptima
e a sua implementa\c{c}\~{a}o num\'{e}rica, para todo e qualquer
$X_0$, podem ser encontrados em \cite{Trelat_book1}.


\section{Controlo \'{o}ptimo n\~{a}o linear}
\label{sec:CO:NL}

Apresentamos agora algumas t\'{e}cnicas para a an\'{a}lise
de problemas de controlo \'{o}ptimo n\~{a}o lineares.
Em particular, formulamos o Princ\'{\i}pio do M\'{a}ximo de Pontryagin
numa forma mais geral do que aquela que vimos
na Sec\c{c}\~{a}o~\ref{sec:cont:linear}.
O exemplo n\~{a}o linear da massa-mola
ser\'{a} tratado como exemplo de aplica\c{c}\~{a}o.


\subsection{Problem\'{a}tica geral}

De um ponto de vista global, o problema deve se formulado
numa variedade $M$, mas o nosso ponto de vista vai ser \emph{local}
e trabalhamos sobre um aberto $V$ de $\R^n$ suficientemente pequeno.
A problem\'{a}tica geral do controlo \'{o}ptimo \'{e} a seguinte.
Consideremos um sistema de controlo
\begin{equation}
\label{contsystnonlin}
\dot{x}(t) = f(x(t), u(t))
\end{equation}
sobre $V$ onde $f: \R^n \times \R^m \to \R^n$
\'{e} ``suave''\footnote{F.H.~Clarke criou nos anos setenta
a chamada ``An\'{a}lise N\~{a}o-Suave'' (\emph{Nonsmooth Analysis})
que permite o estudo de problemas de controlo \'{o}ptimo mais gerais,
em que as fun\c{c}\~{o}es envolvidas n\~{a}o s\~{a}o necessariamente diferenci\'{a}veis
no sentido cl\'{a}ssico. Dado o car\'{a}cter introdut\'{o}rio do nosso texto,
restringimo-nos ao caso ``suave'' no sentido $C^\infty$: todos
os objectos manipulados s\~{a}o aqui, salvo casos particulares mencionados,
$C^{\infty}$. Remetemos o leitor interessado na An\'{a}lise N\~{a}o-Suave
para \cite{Clar83,17,102,24}.} e o conjunto dos controlos admiss\'{\i}veis
$\mathcal{U}$ \'{e} composto por aplica\c{c}\~{o}es
$u : [0, T(u)] \to \Omega \subseteq \R^m$ mensur\'{a}veis limitadas.
Dada uma aplica\c{c}\~{a}o $f^0: \R^n \times \R^m \to \R$,
denotamos por $$C(u) = \int_0^{T'} f^0(x(t), u(t)) dt$$
o custo de uma traject\'{o}ria $x : t \mapsto x(t)$ associada a $u(\cdot)$
e definido sobre $[0, T'(u)]$, $T'(u) \leq T(u)$. Sejam $M_0$ e $M_1$
duas sub-variedades regulares de $V$. O problema do controlo \'{o}ptimo
consiste em encontrar, de entre todas as traject\'{o}rias
que unem $M_0$ a $M_1$, aquelas cujo custo \'{e} m\'{\i}nimo.
Come\c{c}amos por restringir-nos ao caso em que $M_0$ e $M_1$
s\~{a}o pontos $x_0$ e $x_1$ de $V$. Sendo o nosso ponto de vista local,
podemos sempre supor que $x_0 = 0$.


\subsection{Aplica\c{c}\~{a}o entrada-sa\'{\i}da}

Consideremos para o sistema \eqref{contsystnonlin} o seguinte problema
de \emph{controlo}: dado um ponto $x_1 \in V$,
encontrar um tempo $T$ e um controlo $u$ sobre $[0, T]$
tal que a traject\'{o}ria $x_u$ associada a $u$,
solu\c{c}\~{a}o de \eqref{contsystnonlin}, verifica
$$x_u(0) = 0 \, , \hspace{0.2 cm} x_u(T) = x_1 \, .$$
Isto leva-nos a definir:
\begin{definition}
Seja $T>0$. A aplica\c{c}\~{a}o entrada-sa\'{\i}da em tempo $T$
do sistema de controlo \eqref{contsystnonlin}
inicializado em $0$ \'{e} a aplica\c{c}\~{a}o:
\begin{equation*}
\begin{split}
E_T: \,  & \mathcal{U} \to V\\
& u \mapsto x_u(T)
\end{split}
\end{equation*}
onde $\mathcal{U}$ \'{e} o conjunto dos controlos admiss\'{\i}veis.
\end{definition}

Por outras palavras, a aplica\c{c}\~{a}o entrada-sa\'{\i}da em tempo $T$
associa a um controlo $u$ o ponto final da traject\'{o}ria
associada a $u$. Uma quest\~{a}o importante na teoria
do controlo \'{e} estudar esta aplica\c{c}\~{a}o $E_T$,
descrevendo a sua imagem,
as suas singularidades, a sua regularidade, etc.
A resposta a estas quest\~{o}es depende, obviamente,
do espa\c{c}o $\mathcal{U}$ de partida
e da forma do sistema (da fun\c{c}\~{a}o $f$).
Com toda a generalidade temos o seguinte resultado
(\textrm{vide}, \textrm{e.g.}, \cite{Jurdjevic,Sontag}).
\begin{proposition}
Consideremos o sistema \eqref{contsystnonlin} onde $f$
\'{e} ``suave'' e  $\mathcal{U} \subset L^{\infty}([0, T])$.
Ent\~{a}o $E_T$ \'{e} ``suave'' no sentido $L^{\infty}$.
\end{proposition}

Seja $u \in \mathcal{U}$  um controlo de refer\^{e}ncia. Exprimamos
a diferenciabilidade (no sentido de Fr\'{e}chet) de $E_T$ no ponto $u$.
Consideremos $A(t) = \frac{\partial f}{\partial x}(x_u(t), u(t))$
e $B(t) = \frac{\partial f}{\partial u}(x_u(t), u(t))$. O sistema
\begin{equation*}
\begin{split}
\dot{y}_v(t) &= A(t)y_v(t) + B(t) v(t)\\
y_v(0) &= 0
\end{split}
\end{equation*}
\'{e} chamado \emph{sistema linearizado} ao longo de $(x_u, u)$.
O diferencial de Fr\'{e}chet de $E_T$ em $u$ \'{e} a aplica\c{c}\~{a}o
$$dE_T(u) \cdot v = y_v(T) = \int_0^T M(T)M^{-1}(s) B(s)v(s) ds$$
onde $M$ \'{e} a solu\c{c}\~{a}o matricial de $\dot{M} = AM$, $M(0) = Id$.


\subsection{Controlos singulares}

Seja $u$ um controlo definido sobre $[0,T]$ tal que a traject\'{o}ria partindo
de $x(0)=x_0$ \'{e} definida sobre $[0, T]$. Dizemos que o controlo $u$
(ou a traject\'{o}ria $x_u$) \'{e} singular sobre $[0, T]$ se o diferencial
de Fr\'{e}chet $dE_T(u)$ da aplica\c{c}\~{a}o entrada-sa\'{\i}da
no ponto $u$ n\~{a}o \'{e} sobrejectiva. Caso contr\'{a}rio dizemos que $u$ \'{e} regular.

\begin{proposition}
\label{propEaberta}
Sejam $x_0$ e $T$ fixos. Se $u$ \'{e} um controlo regular,
ent\~{a}o $E_T$ \'{e} uma aplica\c{c}\~{a}o aberta numa vizinhan\c{c}a de $u$.
\end{proposition}


\subsection{Conjunto acess\'{\i}vel e controlabilidade}

O conjunto acess\'{\i}vel em tempo $T$ para o sistema \eqref{contsystnonlin},
denotado por $A(T)$, \'{e} o conjunto das extremidades em tempo $T$
das solu\c{c}\~{o}es do sistema partindo de $0$. Por outras palavras,
\'{e} a imagem da aplica\c{c}\~{a}o entrada-sa\'{\i}da em tempo $T$.

\begin{definition}
O sistema \eqref{contsystnonlin} diz-se control\'{a}vel se
$$\cup_{T \geq 0} A(T) = \R^n \, .$$
\end{definition}

Argumentos do tipo do teorema da fun\c{c}\~{a}o impl\'{\i}cita
permitem deduzir os resultados de \emph{controlabilidade local}
do sistema de partida a partir do estudo da controlabilidade
do sistema linearizado (\textrm{vide}, \textrm{e.g.}, \cite{Lee_Markus}).
Por exemplo, deduzimos do teorema de controlabilidade
no caso linear a proposi\c{c}\~{a}o seguinte.
\begin{proposition}
Consideremos o sistema de controlo
\eqref{contsystnonlin} onde $f(0, 0) = 0$.
Seja $A = \frac{\partial f}{\partial x}(0,0)$
e $B = \frac{\partial f}{\partial u}(0,0)$.
Se $$rank \, (B | AB | \cdots | A^{n-1}B) = n $$
ent\~{a}o o sistema n\~{a}o linear
\eqref{contsystnonlin} \'{e} localmente control\'{a}vel em $0$.
\end{proposition}

Em geral o problema da controlabilidade \'{e} dif\'{\i}cil.
Diferentes abordagens s\~{a}o poss\'{\i}veis. Umas fazem uso
da An\'{a}lise, outras da Geometria, outras ainda da \'{A}lgebra.
O problema da controlabilidade est\'{a} ligado, por exemplo,
\`{a} quest\~{a}o de saber quando um determinado semi-grupo opera transitivamente.
Existem tamb\'{e}m t\'{e}cnicas para mostrar, em certos casos,
que a controlabilidade \'{e} global. Uma delas, importante,
\'{e} a chamada ``t\'{e}cnica de alargamento''
(\textrm{vide} \cite{Jurdjevic}).


\subsection{Exist\^{e}ncia de controlos \'{o}ptimos}

Para al\'{e}m de um problema de controlo, consideramos tamb\'{e}m
um problema de optimiza\c{c}\~{a}o: de entre todas as solu\c{c}\~{o}es
do sistema \eqref{contsystnonlin} unindo $0$ a $x_1$,
encontrar uma traject\'{o}ria que minimiza (ou maximiza)
uma certa fun\c{c}\~{a}o \emph{custo} $C(T, u)$. Uma tal traject\'{o}ria,
se existir, diz-se \emph{\'{o}ptima} para esse custo.
A exist\^{e}ncia de traject\'{o}rias \'{o}ptimas depende
da regularidade do sistema e do custo. Para um enunciado
geral de exist\^{e}ncia \textrm{vide}, \textrm{e.g.},
\cite{Jurdjevic,Lee_Markus}. Pode tamb\'{e}m acontecer
que um controlo \'{o}ptimo n\~{a}o exista na classe
de controlos considerada, mas exista num espa\c{c}o mais abrangente.
Esta quest\~{a}o remete-nos para outra \'{a}rea importante:
o estudo da regularidade das traject\'{o}rias \'{o}ptimas.
Francis Clarke e Richard Vinter deram um contributo
important\'{\i}ssimo nesta \'{a}rea, introduzindo o estudo sistem\'{a}tico
da regularidade lipschitziana dos minimizantes
no controlo \'{o}ptimo linear \cite{204,Clarke:Vinter:90,book:Vinter}.
Resultados gerais de regularidade lipschitziana das
traject\'{o}rias minimizantes para sistemas
de controlo n\~{a}o lineares podem ser encontrados
em \cite{mcssPraga}.


\subsection{Princ\'{\i}pio do M\'{a}ximo de Pontryagin}

Dado um problema de controlo \'{o}ptimo para o qual
est\~{a}o garantidas as condi\c{c}\~{o}es de exist\^{e}ncia
e regularidade da solu\c{c}\~{a}o \'{o}ptima,
como determinar os processos optimais?
A resposta a esta quest\~{a}o \'{e} dada
pelo c\'{e}lebre \emph{Princ\'{\i}pio do M\'{a}ximo de Pontryagin}.
Para um estudo aprofundado das condi\c{c}\~{o}es
necess\'{a}rias de optimalidade
sugerimos \cite{Clar83,Smirnov05,Trelat_book1}.

Come\c{c}amos por mostrar que uma traject\'{o}ria singular pode ser
parametrizada como a projec\c{c}\~{a}o de uma solu\c{c}\~{a}o de um
sistema hamiltoniano sujeito a uma \emph{equa\c{c}\~{a}o de restri\c{c}\~{a}o}.
Consideremos o \emph{hamiltoniano} do sistema \eqref{contsystnonlin}:
\begin{equation*}
\begin{split}
H \, : \, \R^n \times \R^n \backslash \{0 \} \times \R^m &\to \R\\
(x, p, u) &\mapsto H(x, p, u) = \langle p, f(x, u)\rangle
\end{split}
\end{equation*}
onde $\langle\, , \, \rangle$ denota o produto escalar usual
de $\R^n$.

\begin{proposition}
\label{propcontsing}
Seja $u$ um controlo singular e $x$ a traject\'{o}ria singular associada
a esse controlo em $[0, T]$. Ent\~{a}o, existe um vector linha cont\'{\i}nuo
$p : [0, T] \to \R^n \backslash \{ 0\}$ tal que as equa\c{c}\~{o}es
seguintes s\~{a}o verificadas para quase todo
o $t \in [0, T]$:
\begin{equation*}
\begin{split}
&\dot{x}(t) = \frac{\partial H}{\partial p}(x(t), p(t), u(t)) \, ,
\hspace{0.2 cm} \dot{p}(t) = - \frac{\partial H}{\partial x}(x(t), p(t), u(t)) \\
&\frac{\partial H}{\partial u}(x(t), p(t), u(t)) = 0 \, \, \,
\text{(equa\c{c}\~{a}o de restri\c{c}\~{a}o)}
\end{split}
\end{equation*}
onde $H$ \'{e} o hamiltoniano do sistema.
\end{proposition}

\begin{proof}
Por defini\c{c}\~{a}o, o par $(x, u)$ \'{e} singular sobre $[0, T]$
se $dE_T(u)$ n\~{a}o \'{e} sobrejectiva. Logo existe um vector
linha $\bar{p} \in \R^n \backslash \{ 0 \}$ tal que
\begin{equation*}
\forall \, v(\cdot) \in L^{\infty}([0, T]) \hspace{0.2 cm}
\langle\bar{p}, dE_T(u) \cdot v\rangle
= \bar{p} \int_0^T M(T) M^{-1}(s) B(s) v(s) ds = 0 \, .
\end{equation*}
Consequentemente,
$$\bar{p} M(T) M^{-1}(s) B(s) = 0  \, \, \text{ em q.t.p. de }
[0, T] \, .$$
Seja $p(t) = \bar{p} M(T)M^{-1}(t)$, $t \in [0, T]$.
Temos que $p$ \'{e} um vector linha de $\R^n \backslash
\{ 0\}$ e $p(T) = \bar{p}$. Diferenciando, obtemos
\begin{equation*}
\dot{p}(t) = - p(t) \frac{\partial f}{\partial x}(x(t), u(t)) \, .
\end{equation*}
Introduzindo o hamiltoniano $H(x, p, u) = \langle
p, f(x, u)\rangle$ conclu\'{\i}mos que
$$\dot{x}(t) = f(x(t), u(t)) = \frac{\partial H}{\partial p}(x(t), p(t), u(t))$$
e
$$\dot{p}(t) =
-p(t) \frac{\partial f}{\partial x}(x(t), u(t))
= - \frac{\partial H}{\partial x}(x(t), p(t), u(t)) \, .$$
A equa\c{c}\~{a}o de restri\c{c}\~{a}o vem de $p(t)B(t) = 0$
pois $B(t) = \frac{\partial f}{\partial u}(x(t), u(t))$.
\end{proof}

\begin{definition}
Ao vector linha $p: [0, T] \to \R^n \backslash \{ 0\}$
da Proposi\c{c}\~{a}o~\ref{propcontsing} chamamos \emph{vector adjunto}
do sistema \eqref{contsystnonlin}.
\end{definition}


\subsubsection{Princ\'{\i}pio do M\'{a}ximo fraco (Teorema de Hestenes)}

Procuramos condi\c{c}\~{o}es necess\'{a}rias de optimalidade.
Consideremos o sistema \eqref{contsystnonlin}.
Os controlos $u(\cdot) \in \mathcal{U}$ s\~{a}o definidos
em $[0, T]$ e tomam valores em
$\Omega = \mathbb{R}^m$ (n\~{a}o existem restri\c{c}\~{o}es aos valores
dos controlos). As traject\'{o}rias associadas devem verificar
$x(0) = x_0$ e $x(T) = x_1$.
O problema consiste em minimizar um custo da forma
\begin{equation}
\label{fcustofraco}
C(u) = \int_0^T f^0(x(t), u(t)) dt \, ,
\end{equation}
onde $f^0 : \R^n \times \R^m \to \R$
\'{e} uma aplica\c{c}\~{a}o $C^{\infty}$ e $T$ est\'{a} fixo.

Associamos ao sistema \eqref{contsystnonlin}
o \emph{sistema aumentado}
\begin{equation}
\label{sistaumentado}
\begin{split}
\dot{x}(t) = f(x(t), u(t))\\
\dot{x}^0(t) = f^0(x(t), u(t))
\end{split}
\end{equation}
e usamos a nota\c{c}\~{a}o $\tilde{x} = (x, x^0)$ e $\tilde{f} = (f, f^0)$.
O problema reduz-se ent\~{a}o \`{a} procura de uma traject\'{o}ria solu\c{c}\~{a}o
de \eqref{sistaumentado} com $\tilde{x}_0 = (x_0, 0)$
e $\tilde{x}_1 = (x_1, x^0(T))$ de tal modo
que a \'{u}ltima coordenada $x^0(T)$ seja minimizada.

Seja $\tilde{x}_0 = (x_0, 0)$ fixo. O conjunto dos estados acess\'{\i}veis
a partir de $\tilde{x}_0$ para o sistema \eqref{sistaumentado}
\'{e} $\tilde{A}(\tilde{x}_0, T) = \cup_{u(\cdot)} \tilde{x}(T, \tilde{x}_0, u) $.
Seja, agora, $u^{*}$ um controlo e $\tilde{x}^{*}$ a traject\'{o}ria associada,
solu\c{c}\~{a}o do sistema aumentado \eqref{sistaumentado} saindo de
$\tilde{x}_0 = (x_0, 0)$. Se $u^{*}$ \'{e} \'{o}ptimo para o crit\'{e}rio
\eqref{fcustofraco}, ent\~{a}o o ponto $\tilde{x}^{*}(T)$ pertence
\`{a} fronteira do conjunto $\tilde{A}(\tilde{x}_0, T)$. Com efeito,
se assim n\~{a}o fosse existiria uma vizinhan\c{c}a do ponto
$\tilde{x}(T) = (x_1, x^0(T))$ em $\tilde{A}(\tilde{x}_0, T)$
contendo um ponto $\tilde{y}^{*}(T)$ solu\c{c}\~{a}o do sistema
\eqref{sistaumentado} e tal que $y^0(T) < x^0(T)$,
o que contradiz a optimalidade do controlo $u^{*}$
(Figura~\ref{accessiblesetaumentado}).
Consequentemente, o controlo $\tilde{u}^{*}$ \'{e},
pela Proposi\c{c}\~{a}o~\ref{propEaberta}, um controlo singular
para o sistema aumentado \eqref{sistaumentado}.

\begin{figure}[!htb]
\begin{center}
\begin{picture}(0,0)%
\includegraphics{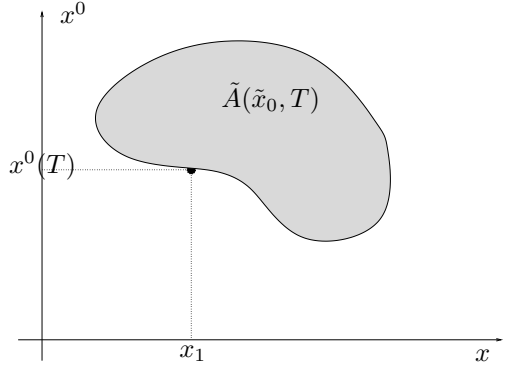}%
\end{picture}%
\setlength{\unitlength}{987sp}%
\begingroup\makeatletter\ifx\SetFigFont\undefined%
\gdef\SetFigFont#1#2#3#4#5{%
  \reset@font\fontsize{#1}{#2pt}%
  \fontfamily{#3}\fontseries{#4}\fontshape{#5}%
  \selectfont}%
\fi\endgroup%
\begin{picture}(12387,8881)(226,-8780)
\put(11926,-8686){\makebox(0,0)[lb]{\smash{\SetFigFont{10}{12.0}{\rmdefault}{\mddefault}{\updefault}{$x$}%
}}}
\put(1501,-211){\makebox(0,0)[lb]{\smash{\SetFigFont{10}{12.0}{\rmdefault}{\mddefault}{\updefault}{$x^0$}%
}}}
\put(4501,-8611){\makebox(0,0)[lb]{\smash{\SetFigFont{10}{12.0}{\rmdefault}{\mddefault}{\updefault}{$x_1$}%
}}}
\put(226,-3961){\makebox(0,0)[lb]{\smash{\SetFigFont{10}{12.0}{\rmdefault}{\mddefault}{\updefault}{$x^0(T)$}%
}}}
\put(5551,-2311){\makebox(0,0)[lb]{\smash{\SetFigFont{10}{12.0}{\rmdefault}{\mddefault}{\updefault}{$\tilde{A}(\tilde{x}_0,T)$}%
}}}
\end{picture}
\end{center}
\caption{se $u^{*}$ \'{e} \'{o}ptimo, ent\~{a}o $\tilde{x}^{*}(T)
\in Fr \, \tilde{A}(\tilde{x}_0, T)$.}
\label{accessiblesetaumentado}
\end{figure}

Usando a Proposi\c{c}\~{a}o~\ref{propcontsing}
obtemos o seguinte teorema.

\begin{theorem}[Princ\'{\i}pio do M\'{a}ximo fraco -- Teorema
de Hestenes \cite{Hestenes}]
\label{teo:Hestenes}
Se $u^{*}$ \'{e} um controlo \'{o}ptimo, ent\~{a}o existe uma aplica\c{c}\~{a}o
$\tilde{p}^{*}:[0, T] \to \R^{n+1} \backslash \{ 0 \}$ tal que
$(\tilde{x}^{*}, \tilde{p}^{*}, \tilde{u}^{*})$
satisfaz o sistema hamiltoniano
\begin{equation}
\label{sisthamilaum1}
\dot{\tilde{x}}^{*}(t) = \frac{\partial \tilde{H}}{\partial
\tilde{p}}(\tilde{x}^{*}(t), \tilde{p}^{*}(t), \tilde{u}^{*}(t)),
\, \, \, \dot{\tilde{p}}^{*}(t) =
- \frac{\partial \tilde{H}}{\partial \tilde{x}}(\tilde{x}^{*}(t),
\tilde{p}^{*}(t), \tilde{u}^{*}(t))
\end{equation}
e a condi\c{c}\~{a}o de estacionaridade
\begin{equation}
\label{sisthamilaum2}
\frac{\partial \tilde{H}}{\partial \tilde{u}}(\tilde{x}^{*}(t),
\tilde{p}^{*}(t), \tilde{u}^{*}(t))=0 \, ,
\end{equation}
onde $\tilde{H}(\tilde{x}, \tilde{p}, u)
= \langle\tilde{p}, \tilde{f}(\tilde{x}, u)\rangle$.
\end{theorem}

O Teorema~\ref{teo:Hestenes} tem a sua g\'{e}nese nos trabalhos de Graves de 1933,
tendo sido obtido primeiramente por Hestenes em 1950 \cite{Hestenes}.
Trata-se de um caso particular do Princ\'{\i}pio do M\'{a}ximo de Pontryagin,
onde n\~{a}o s\~{a}o consideradas restri\c{c}\~{o}es aos valores dos controlos
(\textrm{i.e.}, $u(t) \in \Omega$ com $\Omega = \R^m$).

Escrevendo $\tilde{p}^{*} = (\tilde{p}_1, \ldots, \tilde{p}_n, p_0)
= (\tilde{p}, p_0) \in (\R^n \times \R) \backslash \{ 0\}$,
onde $p_0$ \'{e} a vari\'{a}vel dual do custo e
$\dot{\tilde{p}}^{*}(t) = - \tilde{p}^{*}(t)
\tilde{f}_{\tilde{x}}(\tilde{x}^{*}, u^{*}(t))$,
temos que $(\tilde{p},p_0)$ satisfaz o sistema
\begin{equation*}
(\dot{p}, \dot{p}_0) = - (p, p_0) \left(
\begin{array}{cc} \frac{\partial f}{\partial x} & 0 \\
\frac{\partial f^0}{\partial x} & 0
\end{array} \right)
\end{equation*}
e
\begin{equation*}
\frac{\partial \tilde{H}}{\partial u} = 0
= p \frac{\partial f}{\partial u} + p_0 \frac{\partial f^0}{\partial u}
\end{equation*}
onde $\tilde{H} = \langle\tilde{p}, \tilde{f}(x, u)\rangle
= p \cdot f + p_0 f^0$.
Repare-se que $\dot{p}_0(t)=0$, isto \'{e}, $p_0(t)$ \'{e} constante em $[0, T]$.
Como o vector $p^{*}(t)$ \'{e} definido a menos de uma constante multiplicativa,
escolhe-se normalmente $p_0 \leq 0$.

\begin{definition}
Uma \emph{extremal} do problema de controlo \'{o}ptimo \'{e} um terno ordenado
$(x, p, u)$ solu\c{c}\~{a}o das equa\c{c}\~{o}es \eqref{sisthamilaum1}
e \eqref{sisthamilaum2}. Se $p_0 = 0$, dizemos que a extremal \'{e} anormal.
Nesse caso ela n\~{a}o depende do custo e $(x(t), u(t))$
\'{e} uma traject\'{o}ria singular do sistema \eqref{contsystnonlin}.
\end{definition}

A designa\c{c}\~{a}o \emph{anormal} \'{e} hist\'{o}rica.
Sabe-se hoje que os minimizantes anormais
s\~{a}o frequentes e ``normais'' em muitos e variad\'{\i}ssimos
problemas de optimiza\c{c}\~{a}o. Ao leitor interessado no estudo
de extremais anormais sugerimos o livro \cite{Arutyunov}.


\subsubsection{Princ\'{\i}pio do M\'{a}ximo de Pontryagin}

O princ\'{\i}pio do m\'{a}ximo de Pontryagin
\'{e} uma vers\~{a}o forte do Teorema~\ref{teo:Hestenes}
onde s\~{a}o admitidas restri\c{c}\~{o}es sobre os valores dos controlos.
A exist\^{e}ncia de tais restri\c{c}\~{o}es \'{e} imposta pelas aplica\c{c}\~{o}es
e altera por completo a natureza das solu\c{c}\~{o}es.
O princ\'{\i}pio do m\'{a}ximo de Pontryagin
\'{e} muito mais dif\'{\i}cil de demonstrar do que o
Teorema de Hestenes
(\textrm{vide}, \textrm{e.g.}, \cite{Lee_Markus,Pontryagin_et_all}).
Para uma abordagem simples ao princ\'{\i}pio do m\'{a}ximo de Pontryagin sugerimos
dois livros excelentes escritos em l\'{\i}ngua Portuguesa:
\cite{leitao,Smirnov05}. O enunciado geral \'{e} o seguinte.

\begin{theorem}[Princ\'{\i}pio do M\'{a}ximo de Pontryagin]
\label{thm:PMP}
Considere-se o sistema de controlo em $\R^n$
\begin{equation*}
\dot{x}(t) = f(x(t), u(t)) \, ,
\end{equation*}
onde $f : \R^n \times \R^m \to \R^n$ \'{e} de classe $C^1$
e onde os controlos s\~{a}o aplica\c{c}\~{o}es mensur\'{a}veis e limitadas,
definidos no intervalo $[0, t(u)]$ de $\R$. Denotemos
por $\mathcal{U}$ o conjunto dos controlos admiss\'{\i}veis
cujas traject\'{o}rias associadas unem um ponto inicial
de $M_0$ a um ponto final de $M_1$. Para um tal
controlo definimos o custo
\begin{equation*}
C(u) = \int_0^{t(u)} f^0(x(t), u(t)) dt \, ,
\end{equation*}
onde $f^0 : \R^n \times \R^m \to \R$ \'{e} de classe $C^1$.

Se o controlo $u \in \mathcal{U}$ \'{e} \'{o}ptimo em $[0, t_*]$,
ent\~{a}o existe uma aplica\c{c}\~{a}o n\~{a}o trivial (\textrm{i.e.},
n\~{a}o identicamente nula) $(p(\cdot), p^0) : [0, t_{*}]
\to \R^n \times \R$ absolutamente cont\'{\i}nua, chamada vector adjunto,
onde $p^0$ \'{e} uma constante negativa ou nula, tal que a traject\'{o}ria \'{o}ptima
$x$ associada ao controlo $u$ verifica, em quase todos os pontos de
$[0, t_*]$, o \emph{sistema hamiltoniano}
\begin{equation*}
\dot{x} = \frac{\partial H}{\partial p}(x, p, p^0, u) \, ,
\, \, \, \dot{p} = - \frac{\partial H}{\partial x}(x, p, p^0, u)
\end{equation*}
e a \emph{condi\c{c}\~{a}o do m\'{a}ximo}
\begin{equation*}
H(x(t), p(t), p^0, u(t)) = \max_{v \in \Omega}
H(x(t), p(t), p^0, v) \, ,
\quad \text{q.t.p. } t \in [0, t_*] \, ,
\end{equation*}
onde o \emph{hamiltoniano} $H$ \'{e} dado por
$H(x, p, p^0, u) = \langle p, f(x, u) \rangle
+ p^0 f^0(x, u)$. Al\'{e}m disso, tem-se
para todo o $t \in [0, t_*]$ que
\begin{equation}
\label{condigualzero}
\max_{v \in \Omega} H(x(t), p(t), p^0, v) = 0 \, .
\end{equation}
Se $M_0$ e/ou $M_1$ s\~{a}o variedades de $\R^n$ com espa\c{c}os tangentes
$T_{x(0)}M_0$ em $x(0) \in M_0$ e $T_{x(t_{*})}M_1$
em $x(t_*) \in M_1$, ent\~{a}o o vector adjunto
satisfaz as seguintes condi\c{c}\~{o}es de transversalidade:
$$p(0) \bot T_{x(0)}M_0 \, \, \text{ e } \, \,
p(t_*) \bot T_{x(t_{*})}M_1 \, .$$
\end{theorem}

\begin{remark}
No Teorema~\ref{thm:PMP} o tempo final \'{e} livre. Se impusermos
um tempo final fixo igual a $T$, isto \'{e}, se procuramos,
partindo de $M_0$, atingir o alvo $M_1$ em tempo $T$
e minimizando o custo $C(u)$ em $[0, T]$ (problema a tempo fixo),
ent\~{a}o o teorema continua verdadeiro, salvo a condi\c{c}\~{a}o
\eqref{condigualzero} que deve ser substitu\'{\i}da por
$$
\max_{v \in \Omega} H(x(t), p(t), p^0, v)
= \text{const}  \quad \forall \, t \in [0, T]
$$
(com constante n\~{a}o necessariamente nula).
\end{remark}

\begin{remark}
O problema de tempo m\'{\i}nimo corresponde ao caso em que $f^0 = 1$.
\end{remark}

\begin{remark}
Se o conjunto alvo $M_1$ \'{e} igual a todo o $\R^n$
(problema com extremidade final livre),
ent\~{a}o a condi\c{c}\~{a}o de transversalidade
no instante final diz-nos que $p(t_*) = 0$.
\end{remark}

O princ\'{\i}pio do m\'{a}ximo de Pontryagin \'{e} um resultado profundo
e importante da Matem\'{a}tica contempor\^{a}nea, com in\'{u}meras aplica\c{c}\~{o}es
na F\'{\i}sica, Biologia, Gest\~{a}o, Economia, Ci\^{e}ncias Sociais,
Engenharia, etc. (\textrm{vide}, \textrm{e.g.}, \cite{Bryson}).


\subsection{Exemplo: controlo \'{o}ptimo
de um oscilador harm\'{o}nico (caso n\~{a}o linear)}

Reconsideremos o exemplo (n\~{a}o linear) da mola,
modelado pelo sistema de controlo
\begin{equation*}
\begin{split}
\dot{x}(t) &= y(t) \, ,\\
\dot{y}(t) &= - x(t) -2x(t)^3 + u(t) \, ,
\end{split}
\end{equation*}
onde admitimos como controlos todas as fun\c{c}\~{o}es $u(\cdot)$
seccionalmente cont\'{\i}nuas tais que $|u(t)| \leq 1$.
O objectivo consiste em levar a mola de uma posi\c{c}\~{a}o inicial
qualquer $(x_0, y_0= \dot{x}_0)$ \`{a} sua posi\c{c}\~{a}o de equil\'{\i}brio
$(0, 0)$ \emph{em tempo m\'{\i}nimo} $t_*$.

Apliquemos o Princ\'{\i}pio do M\'{a}ximo de Pontryagin a este problema.
O hamiltoniano tem a forma
$$H(x, y, p_x, p_y, p^0, u)
= p_x y + p_y (-x -2x^3 + u) + p^0 \, .$$
Se $(x, y, p_x, p_y, p^0, u)$ \'{e} uma extremal, ent\~{a}o
\begin{equation*}
\dot{p}_x = -\frac{\partial H}{\partial x}= p_y (1 + 6x^2) \, \,
\text{ e } \, \, \dot{p}_y = -\frac{\partial H}{\partial y} = -p_x \, .
\end{equation*}
Notemos que uma vez que o vector adjunto $(p_x, p_y, p^0)$
deve ser n\~{a}o trivial, $p_y$ n\~{a}o pode anular-se num intervalo
(sen\~{a}o ter\'{\i}amos igualmente $p_x = - \dot{p}_y =0$ e,
por anula\c{c}\~{a}o do hamiltoniano, ter\'{\i}amos tamb\'{e}m $p^0=0$).
Por outro lado, a condi\c{c}\~{a}o do m\'{a}ximo d\'{a}-nos
\begin{equation*}
p_y u = \max_{|v| \leq 1}p_y(t) \, .
\end{equation*}
Em particular, os controlo \'{o}ptimos s\~{a}o sucessivamente iguais
a $\pm 1$, isto \'{e}, verifica-se o princ\'{\i}pio \emph{bang-bang}
(\textrm{vide}, \textrm{e.g.}, \cite{Lee_Markus,MR0638591}).
Concretamente, podemos afirmar que
\begin{equation*}
u(t) = sinal (p_y(t)) \, \, \text{onde} \, \, p_y \, \,
\text{\'{e} a solu\c{c}\~{a}o de} \, \,
\begin{cases}
\ddot{p}_y(t) + p_y(t)(1 + 6x(t)^2) = 0\\
p_y(t_*) = \cos \alpha, \, \dot{p}_y(t_*) = - \sin \alpha \, ,
\end{cases}
\end{equation*}
$\alpha \in [0, 2 \pi[$.
Invertendo o tempo $(t \mapsto -t)$ o nosso problema
\'{e} equivalente ao problema de tempo m\'{\i}nimo para o sistema
\begin{equation*}
\begin{cases}
\dot{x}(t) = - y(t)\\
\dot{y}(t) = x(t) + 2 x(t)^3 -sinal(p_y(t))\\
\dot{p}_y(t) = p_x(t)\\
\dot{p}_x(t) = - p_y(t)(1 + 6x(t)^2) \, .
\end{cases}
\end{equation*}

Dadas as condi\c{c}\~{o}es iniciais $x_0$ e $\dot{x}_0$
(posi\c{c}\~{a}o e velocidade inicial da massa), o problema
\'{e} facilmente resolvido. O leitor interessado
encontra em \cite{Trelat_book1} uma resolu\c{c}\~{a}o
efectuada com o sistema de computa\c{c}\~{a}o alg\'{e}brica
Maple. Sobre o uso do Maple no c\'{a}lculo das varia\c{c}\~{o}es
e controlo \'{o}ptimo veja-se \cite{comPauloLituania05,comAndreia}.


\section*{Nota final}

A Teoria Matem\'{a}tica dos Sistemas e Controlo
\'{e} ensinada nas institui\c{c}\~{o}es dos autores,
nos Departamentos de Matem\'{a}tica
da Universidade de Aveiro
e da Universidade de Orl\'{e}ans, Fran\c{c}a.
Em Aveiro no \^{a}mbito do Mestrado \emph{Matem\'{a}tica e Aplica\c{c}\~{o}es},
especializa\c{c}\~{a}o em \emph{Matem\'{a}tica Empresarial e Tecnol\'{o}gica}
\cite{link:MSc:Aveiro}, e no \^{a}mbito do \emph{Programa Doutoral
em Matem\'{a}tica e Aplica\c{c}\~{o}es} -- este \'{u}ltimo uma associa\c{c}\~{a}o entre
os Departamentos de Matem\'{a}tica da Universidade de Aveiro
e da Universidade do Minho \cite{link:PhD:UA-UM};
em Orl\'{e}ans na op\c{c}\~{a}o ``Controlo Autom\'{a}tico''
do Mestrado PASSION \cite{link:PASSION}.
O primeiro autor foi aluno de Mestrado
em Aveiro e faz actualmente um doutoramento
em Aveiro e Orl\'{e}ans na \'{a}rea do Controlo \'{O}ptimo,
com o apoio financeiro da FCT,
bolsa SFRH/BD/27272/2006.

Agradecemos a um revisor an\'{o}nimo a aprecia\c{c}\~{a}o
cuidada e as numerosas e pertinentes observa\c{c}\~{o}es,
coment\'{a}rios e sugest\~{o}es.



\end{document}